\documentclass[11pt]{article}

\usepackage{amsmath,amssymb}
\usepackage[latin1]{inputenc}
\usepackage[dvips]{graphicx}

\newcommand{\mbN}{{\mathbb{N}}}
\newcommand{\mbZ}{{\mathbb{Z}}}
\newcommand{\mbR}{{\mathbb{R}}}

\newcommand{\mcF}{{\mathcal{F}}}
\newcommand{\mcL}{{\mathcal{L}}}
\newcommand{\mcN}{{\mathcal{N}}}
\newcommand{\mcU}{{\mathcal{U}}}

\newcommand{\cov}{{\mathrm{cov}}}
\newcommand{\var}{{\mathrm{var}}}
\newcommand{\Esp}{{\mathrm{E}}}
\newcommand{\Prob}{{\mathrm{P}}}

\newcommand{\defin}{:=}
\newcommand{\what}{\widehat}
\newcommand{\wtilde}{\widetilde}
\newcommand{\ds}{\displaystyle}

\newtheorem{lemma}{Lemma}
\newtheorem{proposition}{Proposition}
\newtheorem{theorem}{Theorem}
\newenvironment{proof}[1][Proof]{\bigskip\noindent\textbf{#1\, }}{\\}

\begin{document}
 

\noindent
\textbf{\huge  Limiting distributions}

\smallskip
\noindent
\textbf{\huge  for  explosive PAR(1) time series}

\smallskip
\noindent
\textbf{\huge   with strongly mixing 
innovation}

\bigskip
\noindent

\bigskip
\noindent
\textbf{\Large Dominique Dehay}
\footnote{D. Dehay\\
Institut de Recherche Math\'ematique de Rennes,  CNRS umr 6625, 
Universit\'e  de Rennes, France\\
e-mail : 
dominique.dehay@uhb.fr} 

\vspace{10mm} 

\bigskip
\noindent
\textbf{Abstract : }
This work deals with the limiting distribution of the least squares estimators of the coefficients $a_r$ of an explosive periodic autoregressive of order 1 (PAR(1)) time series $X_r=a_rX_{r-1}+u_r$ when the innovation $\{u_k\}$ is   strongly mixing. More precisely
 $\{a_r\}$ is a periodic sequence of real numbers with period $P>0$ and such that $\prod_{r=1}^P|a_r|>1$. The time series $\{u_r\}$ is  periodically distributed with the same period $P$ and  satisfies the strong mixing property, so the random variables $u_r$ can be correlated.
 


\bigskip
\noindent
\textbf{M.S.C 2010:} 62M10, 62M09.


\vspace{10mm}
  
\section{Introduction}
Many man-made  signals and data, even natural ones, exhibit periodicities.  
The non-stationary and seasonal behavior is quite common for many random phenomena as rotating machinery in mechanics (see Antoni 2009), 
seasonal data in econometrics and climatology,  but also signals in communication theory,  biology to name  a few (see e.g. Bloomfield et al. 1994; Chaari et al. 2014; Collet and Martinez 2008; Dragan et al. 1982; Franses and Paap 2004; Gardner et al. 2006; Serpedin et al. 2005; and references therein). 
Periodic autoregressive (PAR) models are one of the simplest linear models with a periodic structure. After more than fifty years of study these models and their generalizations (periodic  ARMA (PARMA), etc.) remain a subject of investigations of great interest as they can  be applied in modeling periodic phenomena for which seasonal ARIMA models do not fit adequatly (see e.g. Adams and Goodwin 1995; Bittanti and Colaneri 2009; Francq et al. 2011; Osborn et al. 1988).

It is well known that such linear models can be represented as   vectorial autoregressive (VAR) models. However the general results known for VAR models do not take into account the whole  periodic structure of the PARMA models (Basawa and Lund 2001; Tia and Grupe 1980) in particular the fact that the innovation can be  periodically distributed. Thus  specific methods have been developped for PAR and PARMA models.

There is a large amount of publications on the estimation problem for the coefficients of PARMA models essentially whenever the model is stable, that is periodic stationary also called cyclostationary (see e.g. Adams and Goodwin 1995; Aknouche and Bibi 2009; Basawa and Lund 2001; Francq et al. 2011; Pagano 1978; Tiao and Grupe 1980; Troutman 1979; Vecchia 1985). 
The unstable case has  been also studied when some autoregressive coefficients are in the boundary of the periodic stationary domain  (see e.g. Aknouche 2012a, 2012b; Aknouche and  Al-Eid 2012; Boswijk and Franses 1995, 1996; Ghyshels and Osborn 2001 and references therein). 

There are few results concerning explosive PAR model. Aknouche (2013) studies the case of  explosive PAR models driven by a periodically distributed independent innovation. However the independence of the innovation is too stringent in practice (see e.g. Aknouche and Bibi 2009; Francq et al. 2011 and references therein). 

In this work we relax the independence condition, and for simplicity of presentation, we consider  periodic autoregressive of order 1 time series that is PAR(1) models. To state the convergence in distribution of the estimators (Theorems~\ref{thrm:conv-loi-|phi|>1} and~\ref{thrm:conv-phi-|phi|>1}) we impose that the periodically distributed innovation is strongly mixing (see e.g. Bradley 2005). Thus it can be correlated and it satisfies some asymptotic independence between its past and its future (see condition (M) in Section~\ref{sect:LSE}). However there is no condition on the rate of the asymptotic independence.  This is
similar to what Phillips (1987) showed for the autoregressive time series with a unit root and constant coefficients.

Here we study the least squares estimators (LSE) of the PAR(1) coefficients, and the limiting distributions stated in Theorem~\ref{thrm:conv-loi-|phi|>1} below generalize the results obtained  by Monsour and Mikulski (1998) for explosive AR models with independent Gaussian innovation (see also Anderson 1959; Stigum 1974) and by  Aknouche (2013) for explosive PAR models with independent innovation. The rate of convergence of the estimators depends on the product  of the periodic PAR(1) coefficients of the PAR(1) model (Theorem~\ref{thrm:conv-loi-|phi|>1}). Actually this product  determines whether the model is stable, unstable or explosive. Thus it is subject of great interest and we  takle the problem of its estimation. For this purpose we consider two estimators : the product of the LSE of the PAR(1) coefficients 
(see (Aknouche 2013) for independent innovation), 
and a least squares estimator (Theorem~\ref{thrm:conv-phi-|phi|>1}). By simulation we detect no specific distinction between these estimators for the explosive PAR(1) models. The theoretical comparison of the limiting distributions of the two estimators is out of the scope of the paper and will be subject to another work.

The paper is organized as follows.
In Section~\ref{sect:background} the model under study is defined as well as the notations. Then the asymptotic behaviour as $n\to\infty$ of the scaled vector-valued time series $\big(\phi^{-n} X_{nP+r}:r=1,\dots,P\big)$ is stated in Section~\ref{sect:asympt-behav} where $\phi$ is the product of the periodic coefficients $a_r$ of the explosive PAR(1) model, thus $|\phi|>1$. The period is $P>0$.  Section~\ref{sect:LSE} deals with the consistency of the least squares estimators $\what{a}_r$ of the PAR coefficients $a_r$, $r=1,\dots,P$, as well as the limiting distributions of the  scaled errors $\phi^n(\what{a}_r-a_r)$, as $n\to\infty$. Next in Section~\ref{sect:estim-phi} we consider the problem of estimation of the product $\phi$. The asymptotic behaviour of the estimators introduced in this paper are illustrated by simulation in Section~\ref{sect:simul}. 
For an easier reading and understanding of the statements of the paper, the proofs of the results are presented in Appendix.

\section{Background~: PAR(1) time series}\label{sect:background}
Consider the following  PAR(1) model 
\begin{equation}\label{model:PAR(1)}
X_k=a_kX_{k-1}+u_k,\qquad k=1,2,\dots,
\end{equation}
where $\{a_k\}$ is a periodic sequence of real numbers and $\{u_k\}$ is a real-valued periodically distributed sequence of centered random variables defined on some underlying probability space  $(\Omega,\mcF,\Prob)$. We assume that the periods of $\{a_k\}$ and of $\{u_k\}$ have the same value $P>0$. Thus $a_{P+r}=a_r$ and $\mcL[u_{P+s},\dots,u_{P+r}]=\mcL[u_s,\dots,u_r]$ for all integers $s$ and $r$. 
To be short, in the sequel the sequence $\{\epsilon_k\}$ is called \emph{innovation} of the model although it is not necessarily uncorrelated.
Denote
$$A_s^{s-1}\defin 1,\qquad  A_s^r\defin\prod_{j=s}^ra_j \quad \mbox{for}\,\,\,1\leq s\leq r,\qquad \mbox{and}\qquad A_s^r\defin 0\quad\mbox{otherwise}$$
and let $\phi\defin A_1^P=\prod_{r=1}^Pa_r$. 
Since $\{a_k\}$ is periodic with period $P>0$, we have 
$A_{1}^{nP+r}=A_1^r\phi^n$ and we obtain the decomposition
\begin{equation}
X_{nP+r}
=A_{1}^{r}X_{nP}+U_n^{(r)}=A_1^r\phi^n\big(X_0+Z_{n-1}\big)+U_n^{(r)}\label{eq:XnP+r}
\end{equation}
where $$U_n^{(r)}\defin\sum_{s=1}^{r}A_{s+1}^{r} u_{nP+s},$$
$Z_{-1}\defin0$ and for $n\geq 1$
\begin{equation*}
Z_n\defin\sum_{l=0}^{n}\phi^{-l-1}\sum_{s=1}^{P}A_{s+1}^{P} u_{lP+s}=\sum_{l=0}^{n}\phi^{-l-1}U_l^{(P)}.
\end{equation*}
Note that the sequence $\big\{\big(U_n^{(1)},\dots,U_n^{(P)}\big):n\in\mbZ\big\}$ is strictly stationary (stationarily distributed).

If $|\phi|<1$, the model is stable and $X_{nP+r}$ converges in distribution to some random variable $\zeta^{(r)}$ as $n\to\infty$. If $|\phi|=1$, the model is unstable : its  behaviour is similar to a random walk; indeed for each $r$, the time series $\{X_{nP+r}\}$ is a random walk when the innovation $\{u_k: k>nP+r\}$ is independent with respect to the random variable $X_{nP+r}$. \\

Henceforth  we assume that the time series $\{X_k\}$ satisfies the PAR(1) equation~(\ref{model:PAR(1)}) with  $|\phi|>1$. We also assume that the initial random variable $X_0$ is square integrable. Moreover 
the innovation $\{u_k\}$ is centered periodically distributed with second order moments.  Thus $\{u_k\}$ is periodically correlated (see Hurd et al. 2002) and we have 
$\Esp[U_n^{(r)}]=0$ as well as
$$\cov\left[U_{n_1}^{(r_1)},U_{n_2}^{(r_2)}\right]
=\sum_{s_1=1}^{r_1}\sum_{s_2=1}^{r_2}A_{s_1+1}^{r_1}A_{s_2+1}^{r_2}\cov\left[u_{s_1}\,,\,u_{(n_2-n_1)P+s_2}\right]
$$
for all integers $n$, $n_1\leq n_2$,  and $r,r_1,r_2=1,\dots,P$.
Denote $\sigma_r\defin\sqrt{\var[u_r]}$ and $K_n^{(r)}\defin
\cov\left[U_{l}^{(r)},U_{l+n}^{(r)}\right]$.

\section{Explosive asymptotic behaviour of the model}\label{sect:asympt-behav}
In the forthcoming proposition we study the asymptotic behaviour of the time series $\{X_k\}$. 
Recall that we assume that $|\phi|>1$.
\begin{proposition}\label{prop:convX-|phi|>1} For any $r=1,\dots,P$
$$\lim_{n\to\infty}\phi^{-n}X_{nP+r}=A_1^r\left(X_0+\zeta\right)\qquad \mbox{a.s. and in q.m.}$$
where
$$\zeta\defin\lim_{n\to\infty}Z_n=\lim_{n\to\infty}\sum_{l=0}^{n}\phi^{-l-1}U_l^{(P)}\qquad\mbox{a.s. and in q.m.} $$
Moreover
$$\lim_{n\to\infty}\mcL\left[X_{nP+r}-\phi^{n}A_1^r\left(X_0+\zeta\right):r=1,\dots,P\right]=\mcL\left[U_0^{(r)}-A_1^r\zeta:r=1,\dots,P\right].$$
\end{proposition}

\noindent
\textbf{Remarks} 

\smallskip\noindent
1) Assume that $\Pr[X_0+\zeta\neq 0]\neq 0$,  then conditionally that $X_0+\zeta\neq 0$, the sequence $\{|X_{k}|\}$ converges to infinity  almost surely as $k\to\infty$. Thus conditionally that $X_0+\zeta\neq 0$, the paths of the time series $\{X_{k}\}$ are  explosive.

\smallskip\noindent
2) When $X_0=-\zeta$\, almost surely,  the time series $\{X_k\}$ which follows  the PAR(1) model with $|\phi|>1$, is  periodically distributed and satisfies the   following stable PAR(1) equation 
$$X_k=a_k^{-1}X_{k-1}+u_k^*$$
where $\{u_k^*\}$ is some periodically distributed time series. 
More precisely $$u_{nP+r}^*=A_1^{r-1}(a_r-a_r^{-1})\sum_{l=0}^{\infty}\phi^{-l-1} U_{n+l}^{(P)}+U_n^{(r-1)}+u_{nP+r}.$$
The estimation problem of the coefficients of such a PAR equation is now well-known, see e.g. (Acknouche and Bibi 2009; Basawa and Lund 2001; Francq et al. 2011  and references therein).
 
\smallskip\noindent
3) To state the convergence in quadratic mean in Proposition~\ref{prop:convX-|phi|>1} we can easily replace the assumption that the innovation $\{u_k\}$ is periodically distributed by the less stringent one that the innovation is periodically correlated.

\smallskip\noindent
4) When the  innovation $\{u_k\}$ is uncorrelated and periodically distributed we obtain
that
$$
\Esp\left[\left(Z_n-\zeta\right)^2\right]=\frac{\phi^{-2n}K_0^{(P)}}{\phi^{2}-1}
\qquad\mbox{and}\qquad \var\left[\zeta\right]=\frac{K_0^{(P)}}{\phi^2-1},$$
with $K_0^{(P)}=\sum_{s=1}^P(A_{s+1}^P)^2\sigma_s$.

\section{Least squares  estimation of the coefficients}\label{sect:LSE}
Now we deal with the estimation problem of the coefficients
$a_r$, $r=1,\dots, P,$ from the observation  $X_k$, $k=0,\dots, nP$, as $n\to\infty$.
%
For that purpose we determine the periodic sequence  $\{b_k\}$ that minimizes the sum of the squared errors
$$\sum_{k=1}^{nP}\big(X_k-b_kX_{k-1}\big)^2=\sum_{r=1}^P\sum_{j=0}^{n-1}\big(X_{jP+r}-b_{r}X_{jP+r-1}\big)^2.$$
Since
$X_{jP+r}=a_{r}X_{jP+r-1}+u_{jP+r}$, for $j=0,\dots,n-1$, and $r=1,\dots,P$,
the least squares estimator (LSE) of $a_r$ is defined by
$$\what{a}_r\defin \frac{\ds\sum_{j=0}^{n-1}X_{jP+r-1}X_{jP+r}}{\ds\sum_{j=0}^{n-1}(X_{jP+r-1})^2}=a_r+\frac{\ds\sum_{j=0}^{n-1}X_{jP+r-1}\,u_{jP+r}}{\ds\sum_{j=0}^{n-1}(X_{jP+r-1})^2}=a_r+\frac{C_n^{(r)}}{B_n^{(r)}},$$
where from expression~(\ref{eq:XnP+r}) we can write
\begin{equation}\label{def:Bnr}
B_n^{(r)}\defin\sum_{j=0}^{n-1}\big(X_{jP+r-1}\big)^2
=\sum_{j=0}^{n-1}\big(A_1^{r-1}\phi^j(X_0+Z_{j-1})+U_j^{(r-1)}\big)^2 
\end{equation}
and
\begin{equation}\label{def:Cnr}
C_n^{(r)}\defin\sum_{j=0}^{n-1}X_{jP+r-1}u_{jP+r}
=A_1^{r-1}\sum_{j=0}^{n-1}\phi^j(X_0+Z_{j-1})u_{jP+r}+\sum_{j=0}^{n-1}U_j^{(r-1)}u_{jP+r}.
\end{equation}
Here $U_j^{(0)}\defin0$.

Note that under  Gaussian and independence assumptions on the periodically distributed innovation $\{u_k\}$, the LSE $\what{a}_r$ coincides with the maximum likelihood estimator of $a_r$.

\bigskip
To prove the convergence in distribution of the scaled errors in the following results we use the next strong mixing condition. The notion of strong mixing, also called $\alpha$-mixing, was introduced by Rosenblatt (1956) and it is largely used for modeling the asymptotic independence in time series. The condition could be weakened using the notion of weak dependence, but this is out of the scope of the paper. For more information about mixing time series and weak dependence see e.g. (Bradley 2005;  Dedeker et al. 2007; Doukhan 1994 and references therein).
  
\medskip\noindent
(\textbf{M})\qquad\emph{ $\ds \lim_{n\to\infty}\alpha(n)=0$\qquad
where $\alpha(n)=\sup\big|\Prob[A\cap B]-\Prob[A]\Prob[B]\big|,$
the supremum being taken over all $k\in\mbN$ and all sets $A\in\mcF^k$, and $B\in\mcF_{k+n}$. Here the $\sigma$-fields $\mcF^k$ and $\mcF_{k+n}$ are defined by $\mcF^k\defin\mcF(X_0,u_j:j\leq k)$ and $\mcF_{k+n}\defin\mcF(u_{j}:j\geq k+n)$.}

\bigskip
Furthermore, in the following we assume that the underlying probability space $(\Omega,\mcF,\Prob)$ is sufficiently large so that there is a sequence of real valued random variables $\{u^*_k\}$ which is independent with respect to $X_0$ and the innovation  $\{u_k\}$, and such that $\mcL\left[u^*_0,\dots, u^*_{nP-1}\right]=\mcL\left[u_{nP},\dots,u_1\right]$ for any integer $n>1$. 
This is always possible at least by enlarging the probability space. 


\bigskip
Now we state the strong consistency of the LSE $\what{a}_r$ of $a_r$, as well as the asymptotic limiting distribution of the scaled error $\phi^n\big(\what{a}_r-a_r\big)$ for the explosive PAR(1)~model~(\ref{model:PAR(1)}).

\begin{theorem}\label{thrm:conv-loi-|phi|>1}
Conditionally that $X_0+\zeta\neq 0$, the least squares estimator $\what{a}_r$ converges to $a_r$ almost surely as $n\to\infty$, for $r=1,\dots,P$.
Furthermore assume that $\Prob[X_0+\zeta=0]=0$ and  the mixing condition~{\rm(\textbf{M})} is fulfilled, then the random vector of the scaled errors $\left\{\phi^{n}\big(\what{a}_r-a_r\big):r=1,\dots,P\right\}$ converges in distribution to $$\left\{\frac{(\phi^2-1)\zeta_r^*}{A_1^{r-1}(X_0+\zeta)}:r=1,\dots,P\right\}$$
 as $n\to\infty$. The random variable $\zeta$ is defined in Proposition~\ref{prop:convX-|phi|>1}, the random vector $\big(\zeta_1^*,\dots,\zeta_P^*\big)$ is independent with respect to $(X_0,\zeta)$,
and its distribution is defined by
$$\mcL\left[\zeta_r^*:r=1\dots,P\right]
=\mcL\left[\sum_{j=1}^{\infty}\phi^{-j}u_{jP-r}^*:r=1,\dots,P\right]$$
where the sequence $\{u_k^*\}$ is independent with respect to  $X_0$ and $\{u_k\}$, and such that  $\mcL\left[u^*_0,\dots, u^*_{nP-1}\right]=\mcL\left[u_{nP},\dots,u_1\right]$ for any integer $n>1$.
\end{theorem}

Note that  in general the limiting distribution of $\phi^n(\what{a}_r-a_r)$ is not parameter free, that is, it depends on the parameters we are estimating.  
Under Gaussian assumption on the periodically distributed innovation $\{u_k\}$, the random variables $\zeta$ and $\zeta_r^*$ are Gaussian and independent, so the distribution of ratio $\zeta_r^*/\zeta$ is a  Cauchy distribution (see also Aknouche 2013). In fact, the limiting distributions in Theorem~\ref{thrm:conv-loi-|phi|>1} have heavy tails.

\section{Estimation of $\phi$}\label{sect:estim-phi}
To estimate $\phi$, the product of the coefficients $a_r$, $r=1,\dots,P$, we can consider the product of the estimators $\what{a}_r$:
$$\wtilde{\phi}=\prod_{r=1}^P\what{a}_r.$$
Then from Theorem~\ref{thrm:conv-loi-|phi|>1}, the estimator $\wtilde{\phi}$ converges almost surely to $\phi$  conditionally that $X_0+\zeta\neq0$. Thanks to the Delta-method (Theorem 3.1, Van der Vaart 1998) we readily deduce the asymptotic law of the normalized error $\phi^n\big(\wtilde{\phi}-\phi\big)$.
\begin{equation*}
\lim_{n\to\infty}\mcL\big[\phi^n\big(\wtilde{\phi}-\phi\big)\big]=\mcL\left[\frac{\phi^2-1}{X_0+\zeta}\times\sum_{r=1}^P A_{r+1}^P\zeta_r^*\right]
\end{equation*}
when we assume that $\Prob[X_0+\zeta=0]=0$ and  the mixing condition~{\rm(\textbf{M})} is fulfilled. See (Aknouche 2013) for independent innovation. 

\medskip
Besides, we can  define a least squares estimator of $\phi$. For that purpose, note that from  relation~(\ref{model:PAR(1)}) and the periodicity of the coefficients, we obtain  for all $j\in\mbN^*$ and $r=1,\dots,P$, that 
$X_{jP+r}=\phi X_{(j-1)P+r}+V_j^{(r)}$ where
$$V_j^{(r)}=\sum_{k=0}^{P-1}A_{r-k+1}^r u_{jP+r-k}.$$
Since the innovation $\{u_k\}$ is periodically distributed with the same period $P$, the sequence $\big\{\big(V_j^{(1)},\dots,V_j^{(P)}\big):j\in\mbZ\big\}$ is strictly stationary (stationarily distributed). Then minimizing  the sum of the squared errors
$$\sum_{k=P+1}^{nP}\big(X_k-b X_{k-P}\big)^2=\sum_{j=1}^{n-1}\sum_{r=1}^P\big(X_{jP+r}-b X_{(j-1)P+r}\big)^2,$$ we define the least squares estimator $\what{\phi}$ as
$$\what{\phi}\defin\frac{\ds\sum_{k=P+1}^{nP}X_kX_{k-P}}{\ds\sum_{k=P+1}^{nP}(X_{k-P})^2}
=\phi+\frac{C_n}{B_n}$$
where
$$C_n=\sum_{j=1}^{n-1}\sum_{r=1}^PX_{(j-1)P+r}V_j^{(r)}\quad
\mbox{and}\quad B_n=\sum_{k=P+1}^{nP}(X_{k-P})^2=\sum_{j=1}^{n-1}\sum_{r=1}^P\big(X_{(j-1)P+r}\big)^2.$$
Then following the same arguments as for the LSE $\what{a}_r$, 
we can state the forthcoming result.
\begin{theorem}\label{thrm:conv-phi-|phi|>1}
Conditionally that $X_0+\zeta\neq 0$, the least squares estimator $\what{\phi}$ converges to $\phi$ almost surely.
Assume that $\Prob[X_0+\zeta=0]=0$ and  the mixing condition~{\rm(\textbf{M})} is fulfilled, then the scaled error $\phi^{n}\big(\what{\phi}-\phi\big)$ converges in distribution 
\begin{equation}\label{lim:whatphi}
\lim_{n\to\infty}\mcL\big[\phi^{n}\big(\what{\phi}-\phi\big)\big]=
\mcL\left[\frac{(\phi^2-1)\zeta^*}{\sum_{r=1}^P\big(A_{r+1}^{P}\big)^{-2}\,(X_0+\zeta)}\right]
\end{equation}
conditionally that $X_0+\zeta\neq 0$, as $n\to\infty$.
Here $\zeta^*$ is a random variable independent with respect to $X_0$ and $\zeta$. The distribution of $\zeta^*$ coincides with the distribution of 
$$\sum_{r=1}^P\sum_{k=0}^{P-1}A_{r+1}^PA_{r-k+1}^r\sum_{j=1}^{\infty}\phi^{-j}u_{jP-r+k}^*.$$
\end{theorem}
In the next section we see by simulation that the distributions of $\tilde{\phi}$ and of $\what{\phi}$  seem to be similar when $|\phi|>1$. The theoretical comparison of these distributions is out of the scope of the paper.

\section{Simulation}\label{sect:simul}
Here we present the simulations of some explosive PAR(1) time series, and we illustrate the behaviour of the LSE $\what{a}_r$, $\what{\phi}$ and of $\tilde{\phi}$ for different values of the coefficients $a_r$, $r=1,\dots,P$ and for different types of innovation. 
For that purpose we consider the PAR(1) model~(\ref{model:PAR(1)}) with period $P=6$. The periodic coefficients $a_r$s are given in Table~\ref{tableau:coeff}.
\begin{table}[ht]\center
\caption{\footnotesize PAR(1) coefficients }
\label{tableau:coeff}
\footnotesize
\begin{tabular}{|r|cccccc|c|}
\hline
& $a_1$ & $a_2$ & $a_3$ & $a_4$ & $a_5$ & $a_6$ & $\phi$\\
\hline
family 1&0.8 & 1.2 & 1   & 1.5 & 1.1 & 0.9 & 1.4256\\
family 2&0.8 & 1.1 & 1   & 1.5 & 1.1 & 0.7 & 1.0164\\
family 3&0.5 & 1   & 1   & 2.5 & 1.6 & 0.5 & 1\\
family 4&0.5 & 1  & 1.5   & 1.62 & 1.6 & 0.5 & 0.972\\
\hline
\end{tabular}
\end{table}

\noindent
Thus we simulate the cases $|\phi|>1$, $|\phi|$ close to 1, $\phi=1$ and $|\phi|<1$. 
The innovation is defined by 
$$u_k=\cos\Big(\frac{\pi k}{3}\Big)\,v_k\qquad\mbox{where}\qquad
v_k=\frac{1}{\sqrt{m+1}}\sum_{i=0}^m\epsilon_{k+i},$$
the random variables $\epsilon_k$, $k\in\mbN$, are independent and identically distributed,  and $m\in\{0\,,\,2000\}$. 
When $m=0$, we have $v_k=\epsilon_k$, $k\in\mbN$, and the random variables $u_k$ are periodically distributed and independent.
When $m=2000$, the random variables $u_k$ are periodically distributed and correlated. Actually they are $m$-dependent, thus strongly mixing.
We consider two  distributions for the $\epsilon_k$s : the standard normal distribution $\mcN(0,1)$, and the uniform distribution $\mcU[-1000,1000]$. Hence in the last case the  distribution of the $\epsilon_k$s is spread out.

To sum up, for each family of coefficients, we obtain four PAR(1) time series that we simulate with different lengths $T=nP=6n$. 
In each case we  perform 100 replications. The algorithm of the simulation is implemented in 'R' software code. Below we present some of the tables and histograms that we obtain  to compare the results.

\begin{table}[ht] \center
\caption{\footnotesize $|\phi|=1.4256$ and uncorrelated Gaussian innovation :  $\mcL[\epsilon_k]=\mcN(0,1)$,    $m=0$,     $n=20$}
\label{tableau:1.4256-N-0-20}
\footnotesize
\begin{tabular}{|l|cccccc|cc|} 
\hline
 \textbf{parameter}&0.8& 1.2& 1& 1.5& 1.1& 0.9& 1.4256&\\
\hline
\textbf{estimate}& $\what{a}_1$ & $\what{a}_2$ & $\what{a}_3$ & $\what{a}_4$  & $\what{a}_5$ & $\what{a}_6$ &  $\what{\phi}$ & $\wtilde{\phi}$\\
mean & 0.8001 & 1.1999 & 0.9999 & 1.4999 & 1.0999 & 0.9000 & 1.4257 & 1.4256\\
median & 0.8000 & 1.1999 & 1.0000 & 1.4999 & 1.0999 & 0.9000 & 1.4256 &  1.4255 \\
\hline
\textbf{error}&&&&&&&&\\
mean& 2e-04 & -4e-05 & -2e-06 & -8e-05 & -1e-05 & 7e-06 & 2e-04 & 9e-05 \\
sigma&  2e-03 & 4e-04 &  8e-06 & 8e-04 & 9e-05 & 9e-05 & 3e-03 & 2e-03\\
\underline{boxplot}&&&&&&&&\\
u. whisker  & 8e-04 & 2e-04 &  8e-07 & 4e-04 &  6e-05 & 6e-05 &  9e-04 &  8e-04\\
u. hinge    & 3e-04 & 5e-05 & 3e-07 & 9e-05 &  2e-05 & 2e-05 & 3e-04 & 3e-14\\
l. hinge  & -2e-04 & -7e-05 & -4e-07 & -2e-04 & -2e-05 &  -2e-05 & -2e-04 & -2e-04\\
l. whisker & 8e-04 & -2e-04 & -9e-07 & -5e-04 & -6e-05 & -5e-05 & -7e-04 & -8e-04\\
\underline{percentile}&&&&&&&&\\
abs 0.95  & 4e-03 & 1e-03 & 7e-06 & 2e-03 & 3e-04 &  3e-04 &  4e-03 &  4e-03\\
\hline
\end{tabular}
\end{table}

\begin{table}[ht] \center
\caption{\footnotesize $|\phi|=1.0164$ and uncorrelated Gaussian innovation :   $\mcL[\epsilon_k]=\mcN(0,1)$,    $m=0$,     $n=200$}
\label{tableau:1.0164-N-0-200}
\footnotesize
\begin{tabular}{|l|cccccc|cc|} 
\hline
 \textbf{parameter}& 0.8 & 1.1 & 1 & 1.5 & 1.1 & 0.7 & 1.0164 &\\
\hline
\textbf{estimate} & $\what{a}_1$ & $\what{a}_2$ & $\what{a}_3$ & $\what{a}_4$  & $\what{a}_5$ & $\what{a}_6$ &  $\what{\phi}$ & $\wtilde{\phi}$\\
mean &   0.7965 & 1.0996 & 0.9999 & 1.4986 & 1.0999 & 0.7000 & 1.0131 &  1.0107 \\
median & 0.7994 & 1.0999 & 1.0000 & 1.4999 & 1.1000 & 0.6999 & 1.0158 &  1.0154 \\
\hline
\textbf{error}&&&&&&&&\\
mean & -4e-03 & -4e-04 & -3e-05 & -2e-03 & -4e-05 & 4e-05 & -4e-03 & -6e-03 \\
sigma & 1e-02 &  2e-03 &  2e-04 &  7e-03 & 5e-04 & 1e-03 & 9e-03 &  3e-02\\
\underline{boxplot}&&&&&&&&\\
u. whisker & 4e-03 &  4e-04 & 2e-05 & 2e-03 & 3e-04 & 9e-04 & 4e-03 &  3e-03\\
u. hinge & 2e-04 & 5e-05 &  5e-06 & 4e-04 & 8e-05 &  3e-04 &  2e-04 & -5e-05\\
l. hinge & -3e-03 & -3e-04 & -4e-06 &-1e-03 & -9e-05 & -3e-04 & -3e-03 & -2e-03\\
l. whisker & -5e-03 & -6e-04 & -2e-05 & -3e-03 & -4e-04 & -8e-04 & -5e-04 & -7e-03\\
\underline{percentile}&&&&&&&&\\
abs 0.95 & 3e-02 &  3e-03 & 4e-04 & 9e-03 & 7e-04 & 2e-03  & 3e-02  & 5e-02\\
\hline
\end{tabular}
\end{table}

\begin{table}[ht] \center
\caption{\footnotesize $|\phi|=1.0164$ and uncorrelated Gaussian innovation :   $\mcL[\epsilon_k]=\mcN(0,1)$,    $m=0$,     $n=400$}
\label{tableau:1.0164-N-0-400}
\footnotesize
\begin{tabular}{|l|cccccc|cc|} 
\hline
 \textbf{parameter}& 0.8 & 1.1 & 1 & 1.5 & 1.1 & 0.7 & 1.0164 & \\
\hline
\textbf{estimate}& $\what{a}_1$ & $\what{a}_2$ & $\what{a}_3$ & $\what{a}_4$  & $\what{a}_5$ & $\what{a}_6$ &  $\what{\phi}$ & $\wtilde{\phi}$\\
mean &   0.7999 & 1.1000 & 1.0000 & 1.5000 & 1.1000 & 0.6999 & 1.0163 & 1.0163 \\
median & 0.8000 & 1.0999 & 1.0000 & 1.4999 & 1.0999 & 0.7000 & 1.0164 & 1.0164\\
\hline
\textbf{error}&&&&&&&&\\
mean  & -8e-05 & 4e-07 & 4e-07 & 5e-06 & 2e-06 & -8e-06 & -7e-05 & -1e-04\\
sigma & 7e-04 & 5e-05 & 3e-06 & 3e-04 & 4e-05 & 1e-04 & 7e-04 & 9e-04\\
\underline{boxplot}&&&&&&&&\\
u. whisker & 2e-04 &  2e-05 &  2e-08 &  8e-05 &  2e-05 &  3e-05 & 2e-04 &  2e-04\\
u. hinge   & 5e-05 &   3e-06 &  5e-09 &  2e-05 &  2e-06 & 2e-05 & 5e-05 &  5e-05\\
l. hinge   & -4e-05 & -8e-06 & -6e-09 & -4e-05 & -4e-06 & -6e-06 & -4e-05 & -5e-05\\
l. whisker & -2e-04 & -2e-05 & -2e-08 & -7e-05 & -2e-05 & -3e-05 & -2e-04 & -2e-04\\
\underline{percentile}&&&&&&&&\\
abs 0.95 & 7e-04 &  8e-05  & 2e-06  & 4e-04 &  5e-05 &  2e-04 &  7e-04 & 8e-04\\
\hline
\end{tabular}
\end{table}

\begin{table}[ht] \center
\caption{\footnotesize $|\phi|=1.0164$ and correlated Gaussian innovation :   $\mcL[\epsilon_k]=\mcN(0,1)$,    $m=2000$,     $n=400$}
\label{tableau:1.0164-N-2000-400}
\footnotesize
\begin{tabular}{|l|cccccc|cc|} 
\hline
\textbf{parameter} & 0.8 & 1.1 & 1 & 1.5 & 1.1 & 0.7 & 1.0164 & \\
\hline
\textbf{estimate}& $\what{a}_1$ & $\what{a}_2$ & $\what{a}_3$ & $\what{a}_4$  & $\what{a}_5$ & $\what{a}_6$ &  $\what{\phi}$ & $\wtilde{\phi}$\\
mean &  0.7999 & 1.1000  &    1.0000 & 1.5000 & 1.1000 & 0.6999 & 1.0163 & 1.0163\\
median & 0.7999 & 1.1000  &    1.0000 & 1.5000 & 1.1000 & 0.6999 & 1.0163 &  1.0163\\
\hline
\textbf{error}&&&&&&&&\\
mean & -8e-05 & 9e-06 & -3e-10 & 5e-05 & 6e-06 & -2e-05 & -8e-05 & -8e-05\\
sigma & 7e-04 & 9e-05 & 2e-09 & 4e-04 & 5e-05 & 2e-04 & 7e-04  & 7e-04\\
\underline{boxplot}&&&&&&&&\\
u. whisker &  2e-04 & 3e-05 & 5e-12 & 2e-04 & 2e-05 &  5e-05 & 2e-04 & 2e-04\\
u. hinge &   5e-05  & 9e-06 & -2e-12 & 4e-05 & 6e-06  & 2e-05  & 5e-05  & 5e-05\\
l. hinge &  -7e-05 & -7e-06 & -2e-11 & -3e-05 & -4e-06 & -2e-05 & -7e-05 & -7e-05\\
l. whisker & -2e-04 & -3e-05 & -4e-11 & -2e-04 & -2e-05 & -5e-05 & -3e-04 & -3e-04\\
\underline{percentile}&&&&&&&&\\
abs 0.95 & 9e-04 &  2e-04 & 5e-10 & 5e-04 &  7e-05 & 2e-04 &  9e-04 & 9e-04\\
\hline
\end{tabular}
\end{table}

\begin{table}[ht]
\center
\caption{\footnotesize $|\phi|=1.0164$ and uncorrelated uniformly distributed innovation  : $\mcL[\epsilon_r]= \mcU[-1000,1000]$, $m=0$, $n=400$}
\label{tableau:1.0164-U-0-400}
\footnotesize
\begin{tabular}{|l|cccccc|cc|} 
\hline
 \textbf{parameter}& 0.8 & 1.1 & 1 & 1.5 & 1.1 & 0.7 & 1.0164 &\\
 \hline
\textbf{estimate}& $\what{a}_1$ & $\what{a}_2$ & $\what{a}_3$ & $\what{a}_4$  & $\what{a}_5$ & $\what{a}_6$ &  $\what{\phi}$ & $\wtilde{\phi}$\\
mean &  0.7991 & 1.0998 & 1.0000 & 1.4996 & 1.0999 & 0.7000 & 1.0156 &  1.0150\\
median & 0.7999 & 1.1000 & 1.0000 & 1.5000 	& 1.1000 & 0.7000 & 1.0163 &  1.0163\\
\hline
\textbf{error}&&&&&&&&\\
mean  & -9e-04 & -2e-04 & 1e-06 & -4e-04 & -2e-05 & 2e-05  & -8e-04 & -2e-03\\
sigma & 6e-03 & 8e-04 & 2e-05 &  3e-03 &  1e-04 & 2e-04 & 6e-03 &  1e-02\\
\underline{boxplot}&&&&&&&&\\
u. whisker & 2e-04 &  3e-05 &  4e-08 &  2e-04 &  2e-05 &  4e-05 & 2e-04 & 2e-04\\
u. hinge  &  4e-05 &  8e-06 & 1e-08 & 4e-05 &  5e-06 & 9e-06 & 4e-05 &    4e-05\\
l. hinge  & -8e-05 & -6e-06 & -4e-10 & -3e-05 & -4e-06 & -2e-05 & -8e-05 & -8e-05\\
l. whisker& -3e-04 & -3e-05 & -4e-05 & -2e-04 & -2e-05 & -4e-05 & -3e-04 & -3e-04\\
\underline{percentile}&&&&&&&&\\
abs 0.95 & 6e-04 &  1e-04 &  7e-06 & 4e-04 &  6e-05 &  2e-04 & 6e-04& 7e-04\\
\hline
\end{tabular}
\end{table}

\begin{table}[ht]
\center
\caption{\footnotesize $|\phi|=1$ and uncorrelated Gaussian innovation :  $\mcL[\epsilon_r]= \mcN(0,1)$, $m=0$, $n=400$}
\label{tableau:1-N-0-400}
\footnotesize
\begin{tabular}{|l|cccccc|cc|} 
\hline
 \textbf{parameter}& 0.5 & 1 & 1 & 2.5 & 1.6 & 0.5 & 1 &  \\
 \hline
\textbf{estimate}& $\what{a}_1$ & $\what{a}_2$ & $\what{a}_3$ & $\what{a}_4$  & $\what{a}_5$ & $\what{a}_6$ &  $\what{\phi}$ & $\wtilde{\phi}$\\
mean   & 0.4948 & 1.0000 &1.0000 & 2.4372 & 1.5972 & 0.5004 & 0.9944 &   0.9644\\
median & 0.4973 & 1.0000 &1.0000 & 2.4566 & 1.5978 & 0.5002 & 0.9972  & 0.9765\\
\hline
\textbf{error}&&&&&&&&\\
mean & -5e-03 &  -2e-06 & 1e-05 & -6e-02 & -2e-03 & 4e-03 & -6e-03 & 4e-02\\
sigma & 9e-03 & 2e-04 & 4e-04 & 6e-02 &  3e-03 & 6e-04 & 9e-03 & 4e-02\\
\underline{boxplot}&&&&&&&&\\
u. whisker & 4e-03 &  8e-05 &  4e-04 & -8e-03 & -4e-04 &  2e-03 &  5e-03 & 2e-03\\
u. hinge   & -4e-04 & 2e-05 &  8e-05 & -3e-02 & -2e-03 &  8e-04 & -5e-04 &  -2e-02\\
l. hinge   & -9e-03 & -2e-05 & -9e-05 & -8e-02 & -4e-03 &  4e-05 & -1e-02 &  -5e-02\\
l. whisker& -2e-02 & -8e-05 & -4e-04 & -2e-01	& -7e-03 & -1e-03 & -3e-02 &  -1e-01\\
\underline{percentile}&&&&&&&&\\
abs 0.95 & 3e-03 &  5e-04 &  9e-04 &  3e-01 & 8e-03 &  2e-03 & 3e-02 & 2e-01 \\
\hline
 \end{tabular}
\end{table}

\begin{table}[ht]
\center
\caption{\footnotesize $|\phi|=0.972$ and  uncorrelated Gaussian innovation :  $\mcL[\epsilon_r]= \mcN(0,1)$, $m=0$, $n=400$}
\label{tableau:0.972-N-0-400}
\footnotesize
\begin{tabular}{|l|cccccc|cc|} 
\hline
\textbf{parameter} & 0.5 & 1 & 1.5 & 1.62 & 1.6 & 0.5 &  0.972 & \\
 \hline
\textbf{estimate}& $\what{a}_1$ & $\what{a}_2$ & $\what{a}_3$ & $\what{a}_4$  & $\what{a}_5$ & $\what{a}_6$ &  $\what{\phi}$ & $\wtilde{\phi}$\\
mean &   0.4938 & 1.0000 & 1.3983 & 1.5579 & 1.5765 & 0.5071 & 0.9662 &  0.8612\\
median & 0.4971 & 1.0000 & 1.4055 & 1.5624 & 1.5777 & 0.5072 & 0.9693 & 0.8745\\
\textbf{error}&&&&&&&&\\
mean  & -7e-03 & 5e-06 & -2e-01 & -7e-02 & -3e-02 & 8e-03 & -6e-03 & -2e-01\\
sigma & 2e-02 & 5e-04 &  3e-02 &  2e-02 &  4e-03 & 1e-03 &  2e-02 &  6e-02\\
\underline{boxplot}&&&&&&&&\\
u. whisker &  2e-02 &  5e-04 & -6e-02 & -4e-02 & -2e-02 & 9e-03 & 2e-02 & -3e-02\\
u. hinge &   3e-03 &  2e-04 & -9e-02 & -6e-02 & -3e-02 & 8e-03 &  4e-03 &-8e-02\\
l. hinge & -2e-02& -2e-04 & -2e-01 & -8e-02 & -3e-02 & 7e-03 & -2e-02 & -2e-01\\
l. whisker & -4e-02 & -5e-04 & -2e-01 & -1e-01 & -4e-02& 7e-03 & -3e-02 & -2e-01\\
\underline{percentile}&&&&&&&&\\
abs 0.95 &  3e-02 &  2e-03 &  2e-01 &  1e-01 & 4e-02 & 9e-03 & 3e-02 & 2e-01\\
\hline
\end{tabular}
\end{table}

\begin{figure}[ht]\center
\includegraphics[height=85mm,width=117mm]{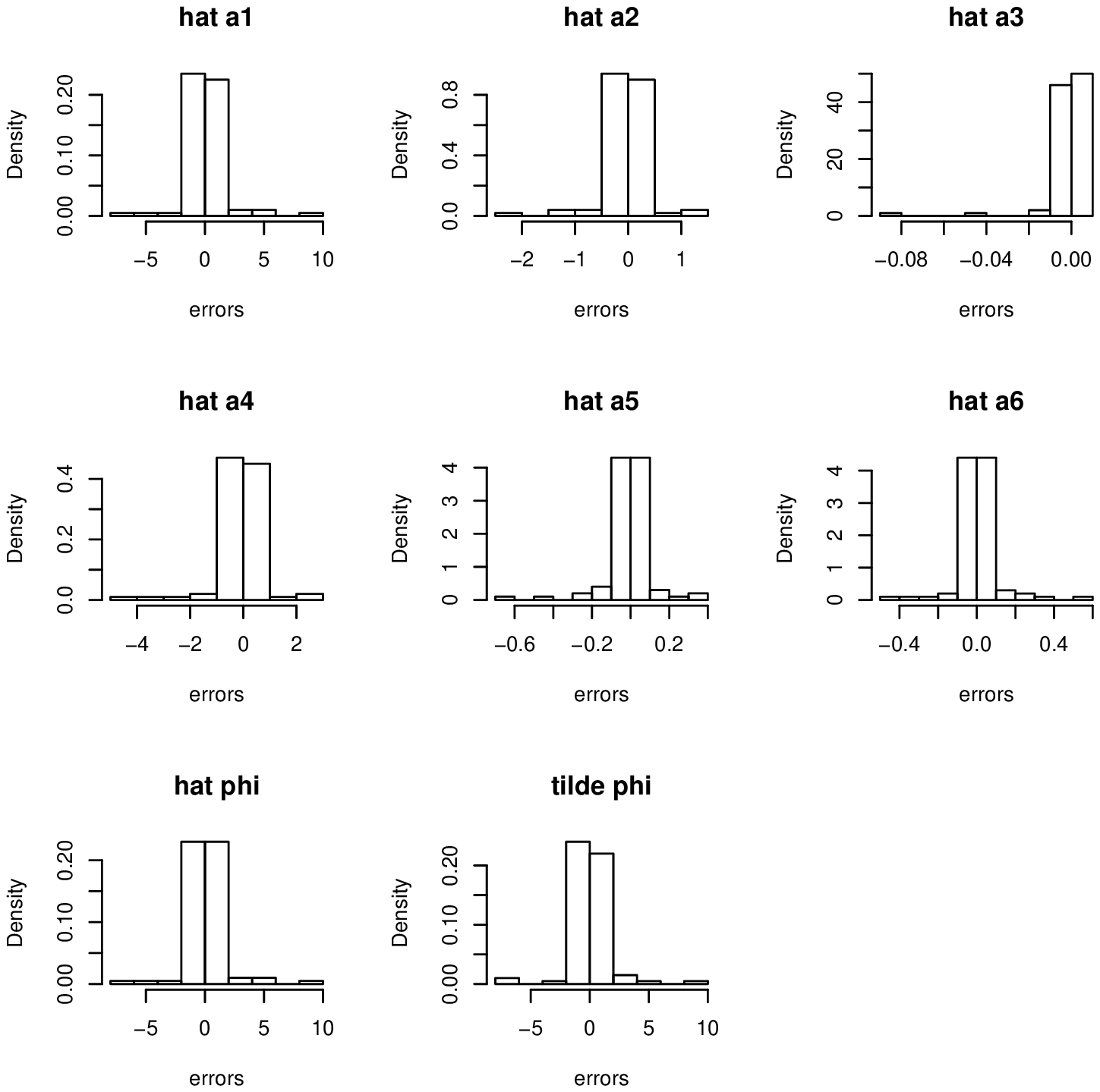}
\caption{\footnotesize Histograms of the scaled errors : $\phi=1.4256$, $\mcL[\epsilon_r]=\mcN(0,1)$, $m=0$, $n=20$}
\label{graph:1.4256-N-0-20}
\end{figure}

\begin{figure}[ht]\center
\includegraphics[height=85mm,width=117mm]{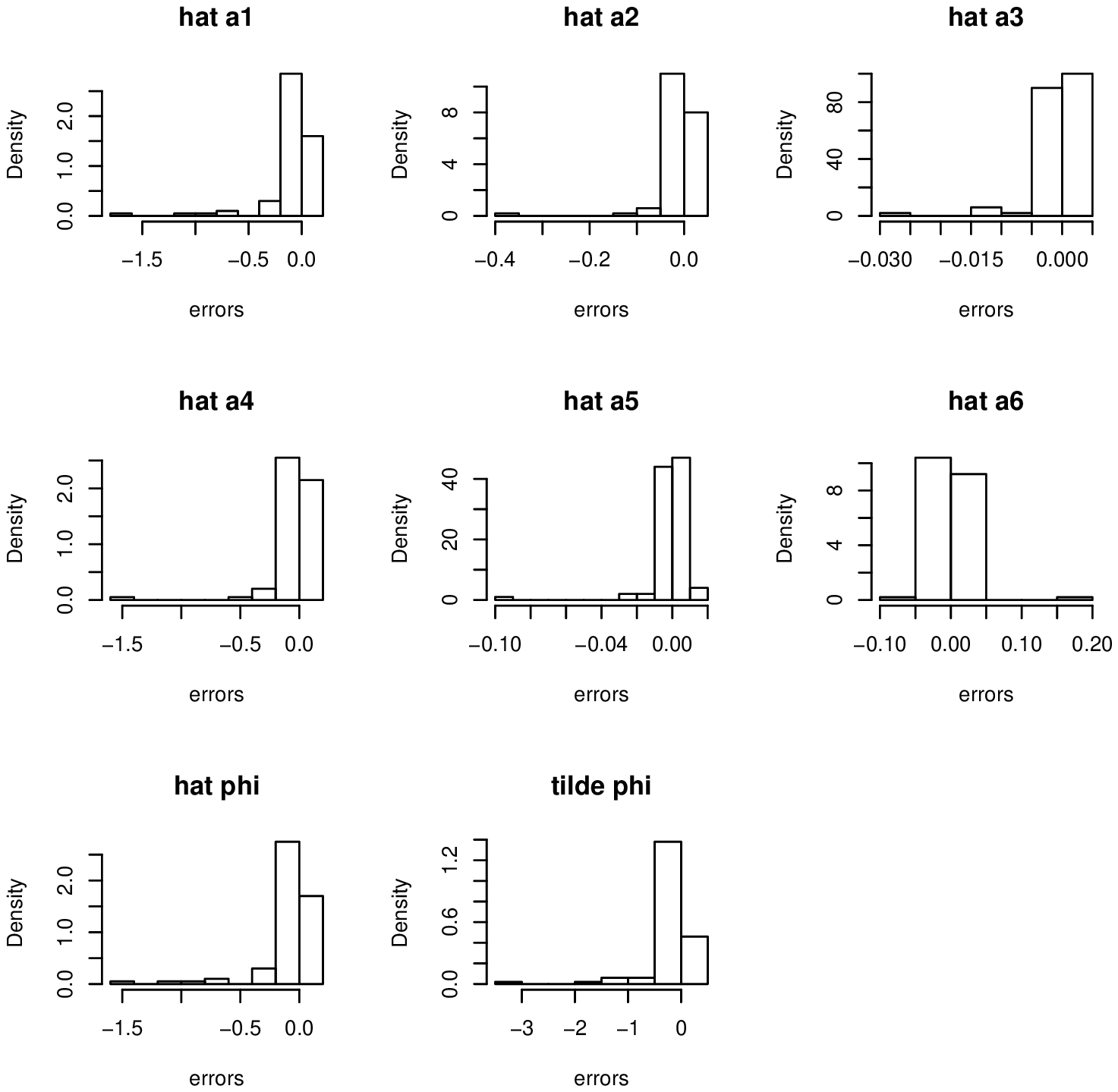}
\caption{\footnotesize Histograms of the scaled errors : $\phi=1.0164$, $\mcL[\epsilon_r]=\mcN(0,1)$, $m=0$, $n=200$}
\label{graph:1.0164-N-0-200}
\end{figure}

\begin{figure}[ht]\center
\includegraphics[height=85mm,width=117mm]{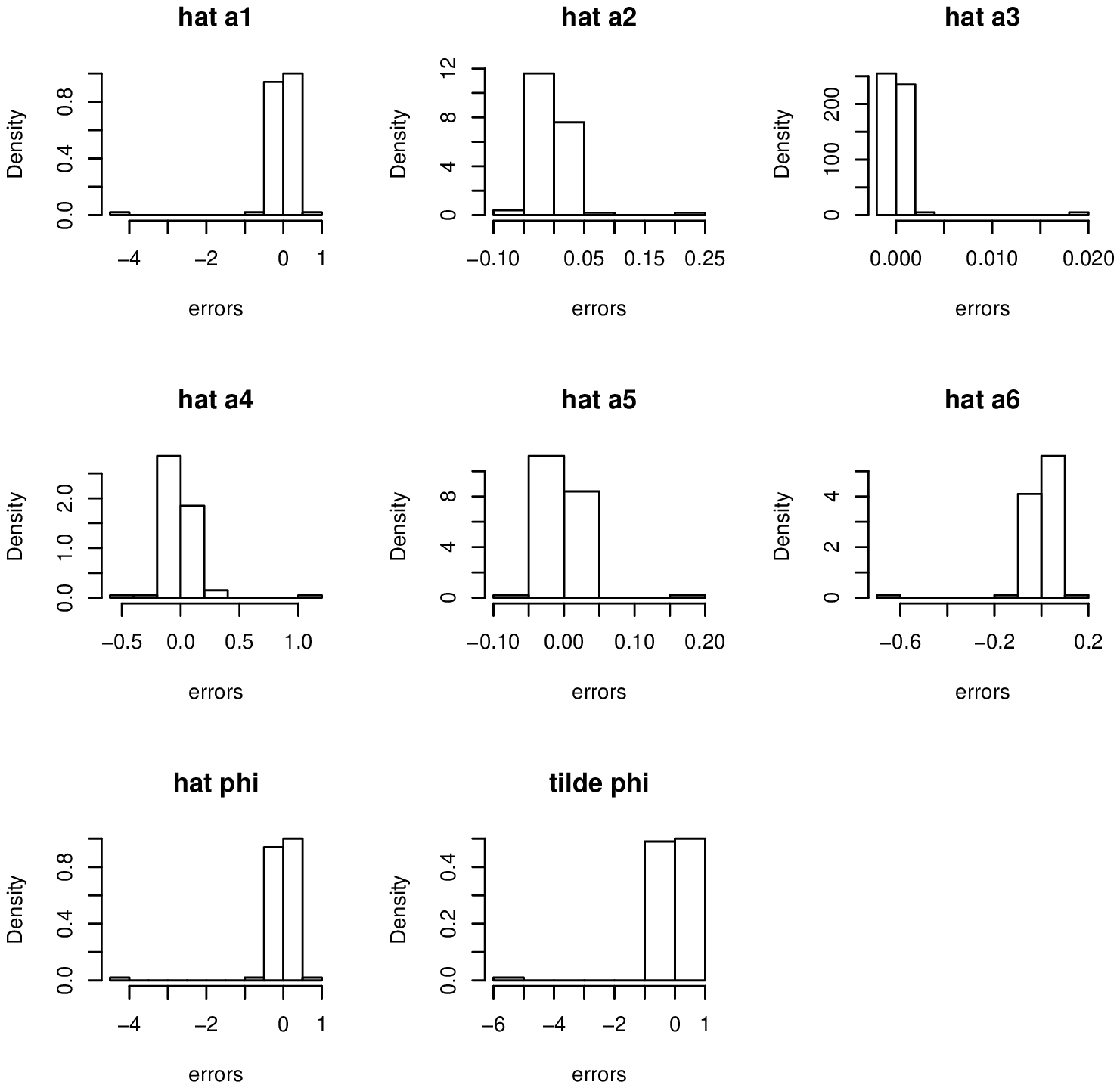}
\caption{\footnotesize Histograms of the scaled errors : $\phi=1.0164$, $\mcL[\epsilon_r]=\mcN(0,1)$, $m=0$, $n=400$}
\label{graph:1.0164-N-0-400}
\end{figure}

\begin{figure}[ht]\center
\includegraphics[height=85mm,width=117mm]{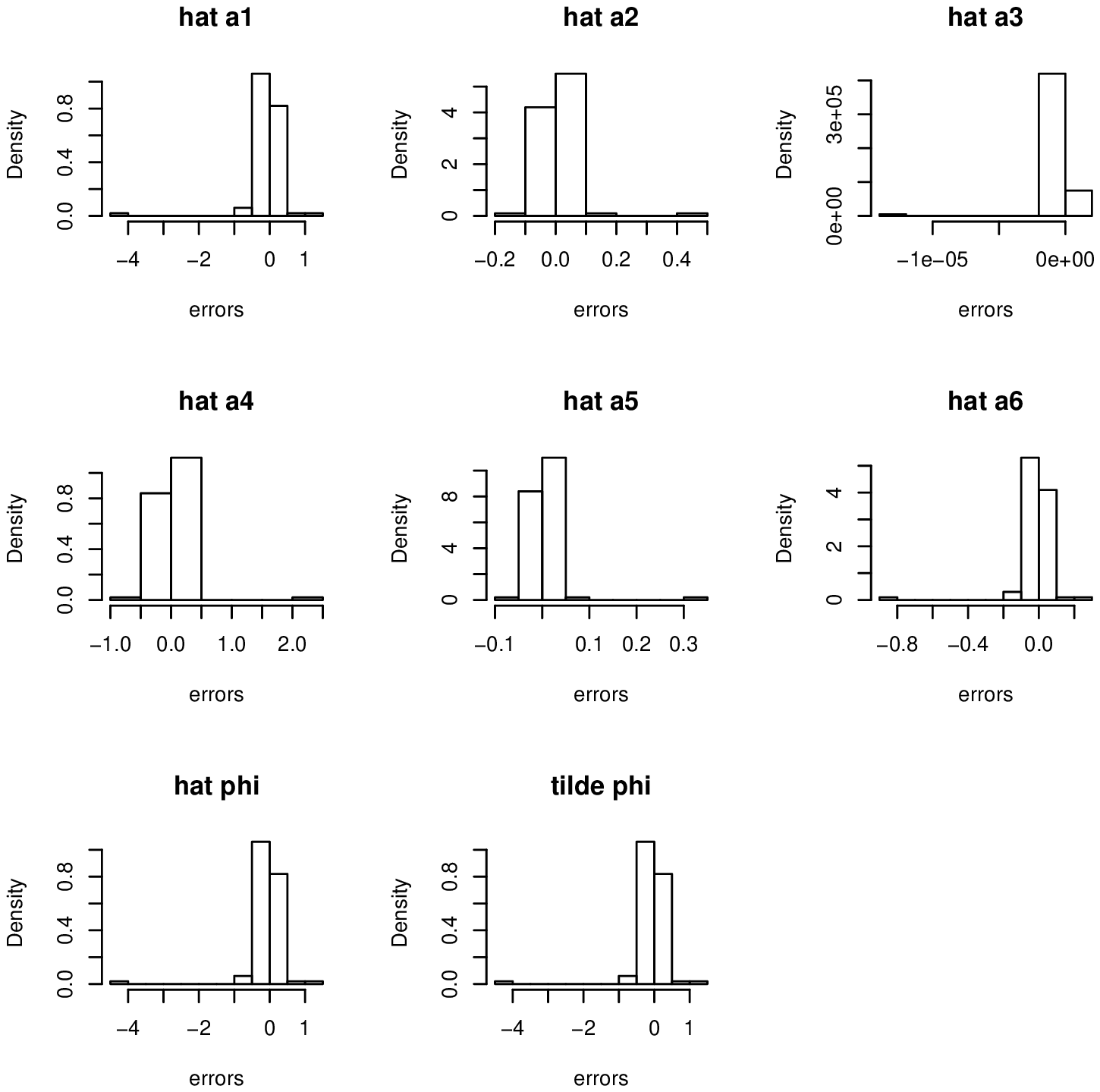}
\caption{\footnotesize Histograms of the scaled errors : $\phi=1.0164$, $\mcL[\epsilon_r]=\mcN(0,1)$, $m=2000$, $n=400$}
\label{graph:1.0164-N-2000-400}
\end{figure}

\begin{figure}[ht]\center
\includegraphics[height=85mm,width=117mm]{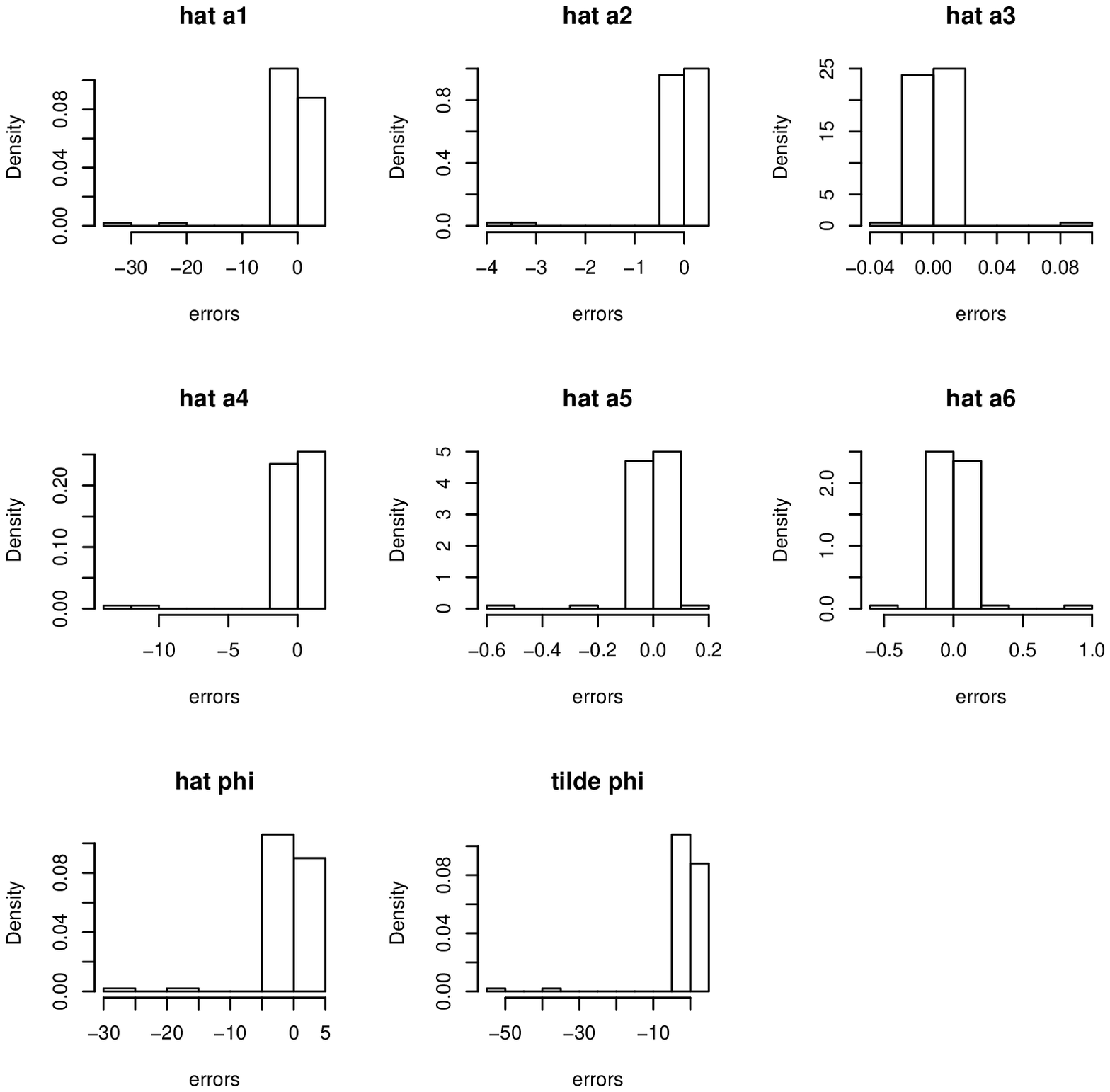}
\caption{\footnotesize Histograms of the scaled errors : $\phi=1.0164$, $\mcL[\epsilon_r]=\mcU[-1000,1000]$, $m=0$, $n=400$}
\label{graph:1.0164-U-0-400}
\end{figure}

\begin{figure}[ht]\center
\includegraphics[height=85mm,width=117mm]{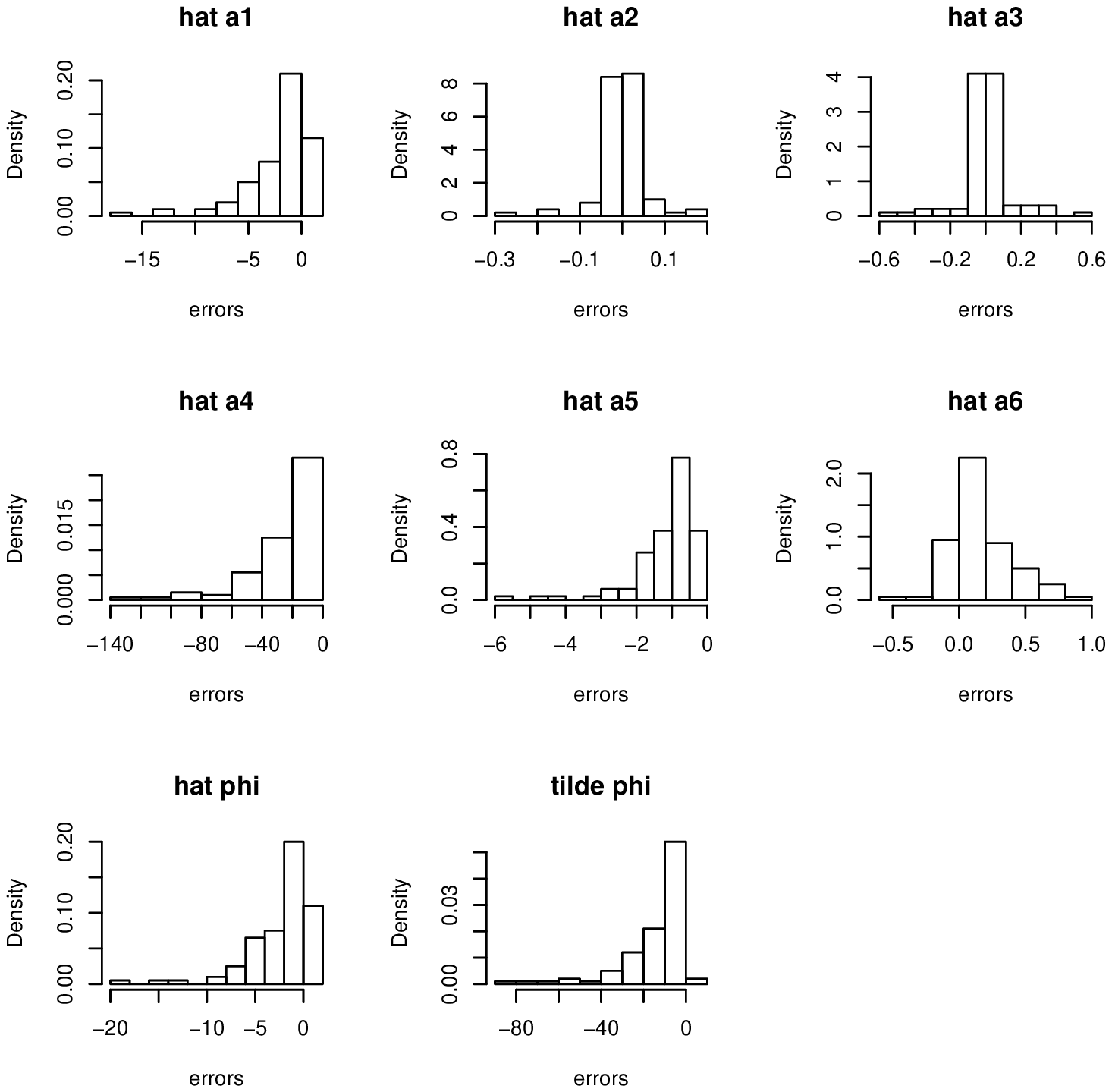}
\caption{\footnotesize Histograms of the scaled errors : $\phi=1$, $\mcL[\epsilon_r]=\mcN(0,1)$, $m=0$, $n=400$}
\label{graph:1-N-0-400}
\end{figure}

\begin{figure}[ht]\center
\includegraphics[height=85mm,width=117mm]{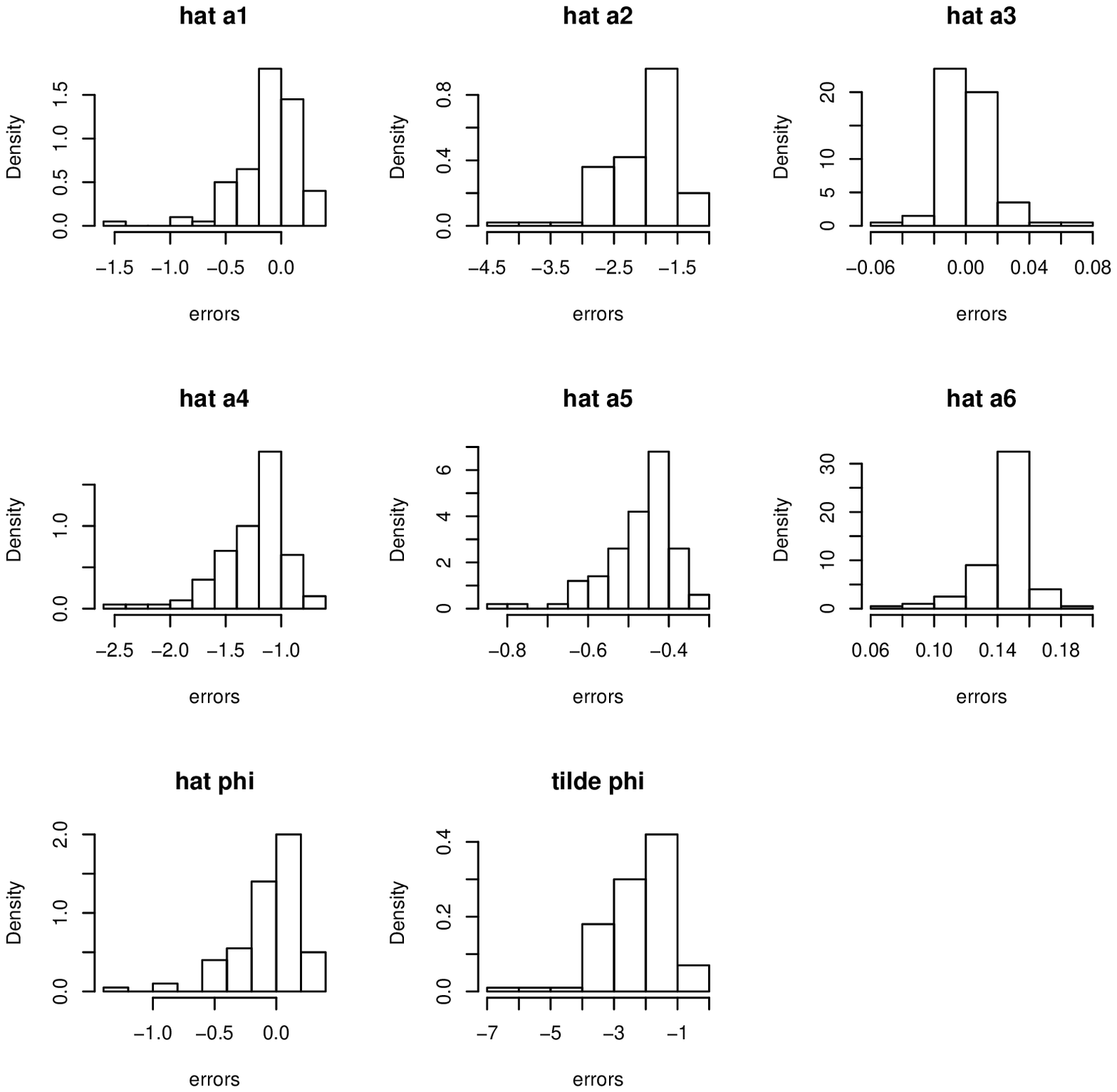}
\caption{\footnotesize Histograms of the scaled errors : $\phi=0.972$, $\mcL[\epsilon_r]=\mcN(0,1)$, $m=0$, $n=400$}
\label{graph:0.972-N-0-400}
\end{figure}

In  each table we write down the mean and the median of the values of each estimator that we have obtained from the 100 replications, as well as some  box-plot characteristics of the  errors :  extrem of upper whishers, upper hinge (3rd quarter), 
lower hinge (1rt quarter) and extrem of the lower whishers. We also give the 95\% percentiles of the absolute values of the errors. 

First  we not that the rates of convergence of the estimates decrease with $|\phi|$ (see Tables~\ref{tableau:1.4256-N-0-20}, \ref{tableau:1.0164-N-0-200}, \ref{tableau:1.0164-N-0-400},  \ref{tableau:1-N-0-400} and \ref{tableau:0.972-N-0-400}). Actually,  from the theoretical point of view, the rate of convergence is $|\phi|^n$ when $|\phi|>1$ (Theorems 1 and 2 above). It is $n$ when $|\phi|=1$ (Boswijk and Franses 1995) and $n^{1/2}$ when $|\phi|<1$ (Basawa and Lund 2001). Thus we produce the histograms of scaled errors, the scale factor being $\phi^{-n}$ when $|\phi|>1$, $n^{-1}$ when $\phi=1$ and $n^{-1/2}$ when $|\phi|<1$. 

We observe when $\phi=1.4256$ and when $\phi=1.0164$ in Figures~\ref{graph:1.4256-N-0-20}, \ref{graph:1.0164-N-0-200}, \ref{graph:1.0164-N-0-400},   \ref{graph:1.0164-N-2000-400} and~\ref{graph:1.0164-U-0-400}, that the histograms of the scaled errors have long tails. 
In Tables~\ref{tableau:1.4256-N-0-20}, \ref{tableau:1.0164-N-0-200}, \ref{tableau:1.0164-N-0-400}, \ref{tableau:1.0164-N-2000-400} and~\ref{tableau:1.0164-U-0-400}, we note that the ratios of the hinges (upper hinge, lower hinge) to the sigmas  of the errors are of order of magnitude $10^{-1}$ or less. It is the same for the ratios of  the hinges to the whiskers.
These phenomena correspond to the fact that the limiting distributions have heavy  tails. 
See (Aknouche 2013) for independent innovation.

When $\phi=1$, in Figure~\ref{graph:1-N-0-400} and in Table~\ref{tableau:1-N-0-400} we observe the distributions for some estimates ($\what{a}_1$, $\what{a}_3$, $\what{a}_6$, $\what{\phi}$) have also relatively long tails. But the phenomenon is very less apparent than previously. See (Phillips 1987) for the autoregressive model with a unit root. 
 
When $\phi=0.972$, in Figure~\ref{graph:0.972-N-0-400} the tails of the histograms are shorter than for the others values of $\phi$. In Table~\ref{tableau:0.972-N-0-400} the hinges and the whishers are often with the same order of magnitude or larger than the sigmas.
This fits to the theoretical result, the limiting distribution being Gaussian (Basawa and Lund 2001).

When the innovation $\{u_k\}$ is correlated with $m=2000$ for $\phi=1.0162$, in Table~\ref{tableau:1.0164-N-2000-400} we observe few change in the performances of the estimates with respect to the case when the innovation is independent (Table~\ref{tableau:1.0164-N-0-400}). The  confidence intervals are smaller for $a_3=1$, and larger for $a_4=1.5$. However in Figure~\ref{graph:1.0164-N-2000-400} with $\phi=1.0164$ and $m=2000$, the tails of the histograms are farther from 0 than in  Figure~\ref{graph:1.0164-N-0-400} when $m=0$. 

Finally, comparing the statistics of the errors of the two estimates $\what{\phi}$ and $\wtilde{\phi}$ in each table, we find out that  they have globally the same order of magnitude whenever $|\phi|>1$ and the histograms are quite similar. However when $|\phi|\leq 1$ it seems that $\what{\phi}$ gives better results than $\tilde{\phi}$. The comparison of these two estimators need more investigation to determine whether one of them is better than the other. 

\section*{Conclusion}
In this paper we have studied the least squares  estimators of the coefficients of  explosive PAR(1) time series under relatively weak dependence assumptions. It is  quite interesting to see how heavy tailed distributions enter in this context.
We have also constructed two estimators of the product of these coefficients, which characterizes the explosive behaviour of the model. It would be worth  to investigate the comparison of these estimators and also to consider more general PAR and PARMA models.

\section*{Appendix}
\subsubsection*{Proof of Proposition~\ref{prop:convX-|phi|>1}}
Let $r=1,\dots,P$ be fixed.
We know that $\phi^{-n}X_{nP+r}=A_1^r(X_0+Z_{n-1})+\phi^{-n}U_n^{(r)}.$
Since the sequence $\{U_n^{(r)}\}$ is stationary with finite second order moments and $|\phi|>1$, we can readily establish that 
$$\lim_{n\to\infty}\phi^{-n}U_n^{(r)}=0\qquad\mbox{a.e and in q.m.}.$$
Now we  show that $Z_n$ converges  almost surely and in quadratique mean.
First we have\\
\begin{eqnarray}
&&\var\left[\sum_{l=j}^{k}\phi^{-l-1}U_{l}^{(P)}\right]\nonumber\\
&&\qquad\qquad=\,
\sum_{l=j}^{k}\phi^{-2l-2}\var\left[U_{l}^{(P)}\right]+2\sum_{l_1=j}^{k-1}\sum_{l_2=l_1+1}^{k}\phi^{-l_1-l_2-2}\cov\left[U_{l_1}^{(P)},U_{l_2}^{(P)}\right]\nonumber\\
&&\qquad\qquad\leq\,
\frac{\phi^{-2j}K_0^{(P)}}{\phi^{2}-1}+2\frac{\phi^{-2j}K_0^{(P)}}{(\phi^2-1)(|\phi|-1)}=\frac{\phi^{-2j}K_0^{(P)}}{(|\phi|-1)^2}.\label{ineq:Sj-cauchy}
\end{eqnarray}
Then the sequence $\{Z_n\}$ is a Cauchy sequence in the Hilbert space $L^2(\Prob)$, thus this sequence converges to some random variable $\zeta$ in quadratic mean. Moreover
\begin{equation*}
\Esp\left[\left(Z_n-\zeta\right)^2\right]\leq
\frac{\phi^{-2n}K_0^{(P)}}{(|\phi|-1)^2}.
\end{equation*}
From the exponential decreasing to 0 of the right hand side of the last inequality, we can readily deduce the convergence almost sure following the usual method applying Borel Cantelli lemma.

As for the second part of the proposition, 
note that
$$X_{nP+r}-\phi^{n}A_1^r\left(X_0+\zeta\right)=
U_n^{(r)}+\phi^{n}A_1^r\left(Z_{n-1}-\zeta\right)$$
and
$$\phi^{n}\left(Z_{n-1}-\zeta\right)=-\sum_{l=n}^{\infty}\phi^{n-l-1}U_{l}^{(P)}=-\sum_{l=0}^{\infty}\phi^{-l-1}U_{n+l}^{(P)}.
$$
Since the sequence of random vectors $\{(U_n^{1},\dots,U_n^{(P)}):n\in\mbZ\}$ is stationarily distributed, we deduce that
$$
\mcL\left[\left(U_n^{(r)},\sum_{l=0}^{\infty}\phi^{-l-1}U_{n+l}^{(r)}\right):r=1,\dots,P\right]=\mcL\left[\left(U_0^{(r)},\sum_{l=0}^{\infty}\phi^{-l-1}U_{l}^{(r)}\right):r=1,\dots,P\right].
$$
Then from the definition of $\zeta$,  we readily achieve the proof of  Proposition~\ref{prop:convX-|phi|>1} .
\hfill{\small$\Box$}

\subsubsection*{Proof of Theorem~\ref{thrm:conv-loi-|phi|>1}}
First in the two following lemmas we study the asymptotic behaviours of $B_n^{(r)}$ and $C_n^{(r)}$.
 
\begin{lemma}\label{lemm:Bn} 
$$\lim_{n\to\infty}\phi^{-2n}B_n^{(r)}=\frac{\big(A_1^{r-1}\big)^2}{\phi^2-1}(X_0+\zeta)^2\qquad\mbox{a.s. and in}\,\, L^1(\Prob).$$
\end{lemma}
\begin{proof}
We have seen that that $\phi^{-n} X_{nP+r-1}$ converges to $A_1^{r-1}(X_0+\zeta)$ almost surely and in quadratic mean. Then Toeplitz lemma on series convergence gives the result. 
\hfill{\small$\Box$}
\end{proof}


\begin{lemma}\label{lemm:Cn}
$$\lim_{n\to\infty}\phi^{-2n}C_n^{(r)}=0\qquad\mbox{a.s. and in}\,\, L^1(\Prob)$$
for any $r=1,\dots,P$.
Moreover, under the mixing hypothesis~{\rm (\textbf{M})} we have
$$\lim_{n\to\infty}\mcL\left[\phi^{-n}C_n^{(r)}:r=1,\dots,P\right]= \mcL\left[A_1^{r-1}(X_0+\zeta)\zeta_r^*:r=1,\dots,P\right].$$
\end{lemma}
\begin{proof}
To prove the lemma we  study the left hand side of equality~(\ref{def:Cnr}).

\medskip
1) For the last term of  equality~(\ref{def:Cnr}), Cauchy Schwarz inequality entails
$$\Esp\left[\left|\sum_{j=0}^{n-1}U_j^{(r-1)}u_{jP+r}\right|\right]
\leq
\sum_{j=0}^{n-1}\Esp\left[\Big|U_j^{(r-1)}\Big|^2\right]^{1/2}\Esp\left[\big|u_{jP+r}\big|^2\right]^{1/2}.$$
As the sequences $\{u_{jP+r}\}$ and $\{U_j^{(r-1)}\}$ are stationary, we have 
$$\Esp\left[\left|\sum_{j=0}^{n-1}U_j^{(r-1)}u_{jP+r}\right|\right]= n\sqrt{K_0^{(r-1)}}\sigma_r.$$
Thus
$$\lim_{n\to\infty}\phi^{-n}\sum_{j=0}^{n-1}U_j^{(r-1)}u_{jP+r}=0\qquad \mbox{in}\,\,L^1(\Prob).$$
Furthermore, thanks to the exponential decreasing rate of convergence to 0 in $L^1(\Prob)$, applying Borel Cantelli lemma, we easily establish the almost sure convergence
$$\lim_{n\to\infty}\phi^{-n}\sum_{j=0}^{n-1}U_j^{(r-1)}u_{jP+r}=0\qquad\mbox{a.s.}$$
 
\smallskip
2) To study the first term of left hand side of equality~(\ref{def:Cnr}) the idea is first to isolate the sums of $Z_{j-1}$ and of $u_{jP+r}$. Then  define blocks that separate the first  and the last terms  of the time series in order to be able to use the asymptotic independence which is given by the strong mixing condition on the innovation.
Thus assume without lost of generality that $n$ is a multiple of 4 and let $n_1=n/4$, and $n_2=n/2$. Then we can write
\begin{eqnarray}
&&\sum_{j=0}^{n-1}\phi^j(X_0+Z_{j-1})u_{jP+r}=
\sum_{j=0}^{n-1}\phi^j(Z_{j-1}-Z_{n_1})u_{jP+r}\nonumber\\
&&\qquad\qquad\qquad+\,
\sum_{j=0}^{n_2-1}\phi^j(X_0+Z_{n_1})u_{jP+r}
+\sum_{j=n_2}^{n-1}\phi^j(X_0+Z_{n_1})u_{jP+r}.\label{equal:nn1n2}
\end{eqnarray}

\smallskip\noindent
(i) 
Thanks to inequality~(\ref{ineq:Sj-cauchy}) we have
\begin{eqnarray*}
&&\Esp\left[\left|\sum_{j=0}^{n-1}\phi^j(Z_{j-1}-Z_{n_1})u_{jP+r}\right|\right]\leq
\sum_{j=0}^{n-1}|\phi|^j\Esp\big[|Z_{j-1}-Z_{n_1}|^2\big]^{1/2}\Esp\big[|u_{jP+r}|^2\big]^{1/2}\\
&&\qquad\quad
\leq\,
\sum_{j=0}^{n_1}|\phi|^j\Esp\big[|Z_{j-1}-Z_{n_1}|^2\big]^{1/2}\sigma_r
+\sum_{j=n_1+1}^{n-1}|\phi|^j\Esp\big[|Z_{j-1}-Z_{n_1}|^2\big]^{1/2}\sigma_r\\
\\
&&\qquad\quad
\leq\,
\frac{n_1+1+|\phi|^{n-n_1}}{|\phi|-1}\sqrt{K_0^{(P)}}\sigma_r.
\end{eqnarray*}
Hence
$$\lim_{n\to\infty}\phi^{-n}\sum_{j=0}^{n-1}\phi^j(Z_{j-1}-Z_{n_1})u_{jP+r}=0\qquad\mbox{in}\,\,L^1(\Prob).$$
Using Borel Cantelli lemma, the exponential decreasing rate of convergence permits to prove the almost sure convergence.

\medskip
\noindent
(ii)  Besides, the second term of the right hand side of equality~(\ref{equal:nn1n2}) can be estimated as follows 
\begin{eqnarray*}
&&\Esp\left[\left|(X_0+Z_{n_1})\sum_{j=0}^{n_2-1}\phi^ju_{jP+r}\right|\right]
\leq
\Esp\big[|X_0+Z_{n_1}|^2\big]^{1/2}\sum_{j=0}^{n_2-1}|\phi|^j\Esp\big[|u_{jP+r}|^2\big]^{1/2}\\
&&\qquad\quad\leq\,
\left(\Esp\big[|X_0|^2\big]^{1/2}+\sum_{j=0}^{n_1-1}|\phi|^{-l-1}\Esp\big[|U_{l}^{(P)}|^2\big]^{1/2}\right)\left(\sum_{j=0}^{n_2-1}|\phi|^j\Esp\big[(u_{jP+r})^2\big]^{1/2}\right)\\
&&\qquad\quad\leq\,
\frac{\Big(\Esp\big[|X_0|^2\big]^{1/2}(|\phi|-1)+\sqrt{K_0^{(P)}}\Big)\sigma_r|\phi|^{n_2}}{(|\phi|-1)^2}.
\end{eqnarray*}
Thus 
$$\lim_{n\to\infty} \phi^{-n}(X_0+Z_{n_1})\sum_{j=0}^{n_2-1}\phi^ju_{jP+r}
=0\qquad\mbox{in}\,\,L^1(\Prob).$$
As in part~(i), we  obtain the almost sure convergence.

\medskip
\noindent
(iii)
It remains to study the asymptotic behaviour of
$(X_0+Z_{n_1})\Psi_{n_2}^{n,r}$ where
\begin{equation}\label{Psir}
\Psi_{n_2}^{n,r}\defin\phi^{-n}\sum_{j=n_2}^{n-1}\phi^j u_{jP+r}=\sum_{j=n_2}^{n-1}\phi^{j-n} u_{jP+r}=\sum_{j=1}^{n-n_2}\phi^{-j} u_{(n-j)P+r}.
\end{equation}
We know that $X_0+Z_{n_1}$ converges to $X_0+\zeta$ almost surely and in quadratic mean.  (Proposition~\ref{prop:convX-|phi|>1}). 
Since
$$\Esp\big[|\Psi_{n_2}^{n,r}|^2\big]^{1/2}\leq
\sum_{j=1}^{n-n_2}|\phi|^{-j} \Esp\big[|u_{(n-j)P+r}|^2\big]^{1/2}
\leq \frac{\sigma_r}{|\phi|-1},$$
$\phi^{-n}\Psi_{n_2}^{n,r}$ converges to 0 in quadratic mean and also almost surely.
Hence $\phi^{-n}(X_0+Z_{n_1})\Psi_{n_2}^{n,r}$ converges to 0 in $L^1(\Prob)$ and  almost surely.

\medskip
\noindent
(iv) Now we establish the convergence in distribution of $(X_0+Z_{n_1})\Psi_{n_2}^{n,r}$.
Note that $(X_0+Z_{n_1})$ can be expressed with $X_0,u_1,u_2,\dots,u_{(n_1+1)P}$ while $\Psi_{n_2}^{n,r}$ can be expressed with $u_{n_2P},\dots,u_{nP}$.
Hence, as $n_1=n_2/2=n/4$, the mixing property entails that 
$$\Big|\Prob\big[X_0+Z_{n_1}\in A,\Psi_{n_2}^{n,r}\in B\big]-\Prob\big[X_0+Z_{n_1}\in A\big]\Prob\big[\Psi_{n_2}^{n,r}\in B\big]\Big|\leq\alpha\big((n/4-1)P\big)$$
for all Borel subsets $A$ and $B$ of $\mbR$, where $\alpha(\cdot)$ is the strong mixing coefficient. The mixing hypothesis entails  that $\alpha\big((n/4-1)P)\big)$ tends to 0 as $n$ goes to $\infty$, thus
$$\lim_{n\to\infty}\Big(\Prob\big[X_0+Z_{n_1}\in A,\Psi_{n_2}^{n,r}\in B\big]-\Prob\big[X_0+Z_{n_1}\in A\big]\times\Prob\big[\Psi_{n_2}^{n,r}\in B\big]\Big)=0.$$
for all Borel subsets $A$ and $B$. So $X_0+Z_{n_1}$ and $\Psi_{n_2}^{n,r}$ are asymptotically independent.
We know that $X_0+Z_{n_1}$ converges in quadratic mean so in distribution to $X_0+\zeta$ (Proposition~\ref{prop:convX-|phi|>1}). 
Now it remains to study the behaviour of  $\Psi_{n_2}^{n,r}$.
 
\medskip\noindent
(v) 
Since the time series $\{u_k\}$ is periodically distributed, the time series $\{u_k^*\}$ is also periodically distributed. Denoting
$$\Psi_n^{*(r)}\defin\sum_{j=1}^{n}\phi^{-j} u_{jP-r}^*,$$
we have  $\mcL[\Psi_{n-n_2}^{*(r)}]=\mcL[\Psi_{n_2}^{n,r}]$, and the sequence $\{\Psi_n^{*(r)}\}$ converges almost surely and in quadratic mean to some random variable $\zeta_r^*$. Then $\Psi_{n_2}^{n,r}$ converge in distribution to $\zeta_r^*$ as $n-n_2=n/2\to\infty$. 
Consequently $(X_0+Z_{n_1})\Psi_{n_2}^{n,r}$ converges in distribution to $(X_0+\zeta)\zeta^{*(r)}$ where $X_0+\zeta$ and $\zeta^{*(r)}$ are independent random variables. Furthermore from definition~(\ref{Psir}), we easily deduce the distribution of the $\zeta_r^*$s.

Following the same lines with Cram\'er device we can establish the multidimensional convergence. 
\hfill{\small$\Box$}
\end{proof}

\begin{proof} \!\textbf{of Theorem~\ref{thrm:conv-loi-|phi|>1}}. The almost sure convergence is a direct consequence of Lemmas~\ref{lemm:Bn} and~\ref{lemm:Cn}.
To prove the convergence in distribution, we first apply Cram\'er device to prove the convergence in distribution of $\big(\phi^{-2n}B_n^{(1)},\dots, \phi^{-2n}B_n^{(P)},\phi^{-n}C_n^{(1)},\dots,\phi^{-n}C_n^{(P)}\big)$. 
For this purpose, let $\alpha_1,\dots,\alpha_P,\beta_1,\dots,\beta_P\in\mbR$
and establish the convergence in distribution of $$
S_n\defin\sum_{r=1}^P\alpha_r\phi^{-2n}B_n^{(r)}+\sum_{r=1}^P\beta_r\phi^{-n}C_n^{(r)}.
$$
Following the same method as in the proof of Lemma~\ref{lemm:Cn} we define blocks to separate the  terms $B_n^{(r)}$ and $C_n^{(r)}$, as well as to apply the asymptotic independence given by the strong mixing condition. Denote $n_1=n_2/2=n/4$. Then $S_n$ can be decomposed as
$$S_n
=\sum_{r=1}^P\alpha_r\big(\phi^{-2n}B_n^{(r)}-\phi^{-2n_1}B_{n_1}^{(r)}\big)+\phi^{-2n_1}\sum_{r=1}^P\alpha_rB_{n_1}^{(r)}+\phi^{-n}\sum_{r=1}^P\beta_rC_n^{(r)}.$$

Thanks to Lemma~\ref{lemm:Bn}, 
$\phi^{-2n}B_n^{(r)}-\phi^{-2n_1}B_{n_1}^{(r)}$ converges to 0 almost surely.
Moreover thanks to the parts 1,  2(i) and 2(ii) of the proof of Lemma~\ref{lemm:Cn}, it remains to study 
$$\sum_{r=1}^n\alpha_r\phi^{-2n_1}B_{n_1}^{(r)}+\sum_{r=1}^n\beta_r(X_0+Z_{n_1})\Psi_{n_2}^{n,r}.$$
Then from the strong mixing condition,  $\big(B_{n_1}^{(1)},\dots,B_{n_1}^{(P)},X_0+Z_{n_1}\big)$ is asymptotically independent with respect to $\big(\Psi_{n_2}^{n,1},\dots,\Psi_{n_2}^{n,P}\big)$, 
and following the same lines as in the  part 2(v) of the proof of Lemma~\ref{lemm:Cn} we deduce the convergence in distribution of $S_n$ as $n\to\infty$.
 
Finally  the application of the continuous mapping theorem for convergence in distribution completes the proof of the theorem. 
\hfill{\small$\Box$}
\end{proof}

\subsubsection*{Proof of Theorem~\ref{thrm:conv-phi-|phi|>1}}
Theorem~\ref{thrm:conv-phi-|phi|>1} is a direct consequence of the following lemma about the asymptotic behaviours of $B_n$ and $C_n$.
\begin{lemma}
\begin{equation}\label{lim:Bn}
\lim_{n\to\infty}\phi^{-2n}B_n=\sum_{r=1}^P\big(A_{r+1}^{P}\big)^{-2}\,\frac{(X_0+\zeta)^2}{\phi^2-1}\qquad\mbox{a.s. and in}\,\, L^1(\Prob);
\end{equation}
\begin{equation*}
\lim_{n\to\infty}\phi^{-2n}C_n=0\qquad\mbox{a.s. and in}\,\, L^1(\Prob);
\end{equation*}
and
\begin{equation}\label{lim:Cn}
\lim_{n\to\infty}\mcL\big[\phi^{-n}C_n\big]=\mcL\big[(X_0+\zeta)\zeta^*\big].
\end{equation}
\end{lemma}

\begin{proof}
1) First note that $B_n$ can be expressed as follows
\begin{eqnarray*}
B_n
=\sum_{r=2}^PB_{n-1}^{(r)}+B_{n}^{(P)}-(X_0)^2.
\end{eqnarray*}
Then thanks to Lemma~\ref{lemm:Bn} 
$$\lim_{n\to\infty}\phi^{-2n}B_n= \left(1+\phi^{-2}\sum_{r=2}^P\big(A_1^{r-1}\big)^2\right)\,\frac{(X_0+\zeta)^2}{\phi^2-1}\qquad\mbox{a.s. and in}\,\, L^1(\Prob).$$
Using the fact that $\phi^{-1}A_1^{r-1}=\big(A_r^P\big)^{-1}$, we deduce  limit~(\ref{lim:Bn}).

\bigskip
2) The  convergence almost sure and in $L^1(\Prob)$ of $\phi^{-2n}C_n$ to 0, can be easily obtained following the lines  of the proof of the convergence almost sure and in $L^1(\Prob)$ of $\phi^{-2n}C_n^{(r)}$ in Lemma~\ref{lemm:Cn}. Thus the proof is left to the reader.

\bigskip
3)  From its definition, $C_n$ can be expressed by
\begin{eqnarray}
C_n
=
\sum_{r=1}^PA_1^r\sum_{j=0}^{n-2}\phi^j\big(X_0+Z_{j-1}\big)V_{j+1}^{(r)}+
\sum_{r=1}^P\sum_{j=0}^{n-2}U_j^{(r)}V_{j+1}^{(r)}\label{rel:Cn}.
\end{eqnarray}

\smallskip\noindent
(i) Thanks to the stationarity of the sequences $\big\{\big(U_j^{(1)},\dots,U_j^{(P)}\big):j\in\mbZ\big\}$ and   $\big\{\big(V_j^{(1)},\dots,V_j^{(P)}\big):j\in\mbZ\big\}$, the second term of expression~(\ref{rel:Cn}) is of order of magnitude $n$ in probability. Indeed
$$\Esp\left[\left|\sum_{r=1}^P\sum_{j=0}^{n-2}U_j^{(r)}V_{j+1}^{(r)}\right|\right]
\leq
(n-1)\sum_{r=1}^P\Esp\left[\big(U_0^{(r)}\big)^2\right]^{1/2}\Esp\left[\big(V_{0}^{(r)}\big)^2\right]^{1/2}.$$

\smallskip\noindent
(ii) Following the same lines as in the proof of Lemma~\ref{lemm:Cn}, and using the stationarity of $\big\{\big(V_j^{(1)},\dots,V_j^{(P)}\big):j\in\mbZ\big\}$ we can readily prove that
$$\phi^{-n}\sum_{r=1}^PA_1^r\sum_{j=0}^{n-2}\phi^j\big(X_0+Z_{j-1}\big)V_{j+1}^{(r)}=
\big(X_0+Z_{n_1}\big)\Psi_{n_2}^n
+o_{\Prob}(1)$$
where $n_1=n_2/2=n/4$ and
\begin{equation}\label{Psi}
\Psi_{n_2}^n\defin\sum_{r=1}^PA_1^r\sum_{j=n_2}^{n-2}\phi^{j-n}V_{j+1}^{(r)}=\sum_{r=1}^PA_1^r\sum_{j=2}^{n-n_2}\phi^{-j}V_{n-j+1}^{(r)}.
\end{equation}
The mixing hypothesis~{\rm(\textbf{M})} entails that $X_0+Z_{n_1}$ and $\Psi_{n_2}^n$ are asymptotically independent. Besides the distribution of
$$\Psi_{n_2}^n =\sum_{r=1}^P\sum_{k=0}^{P-1}A_1^rA_{r-k+1}^r\sum_{j=2}^{n-n_2}\phi^{-j}u_{(n-j+1)P+r-k}.$$
coincides with the distribution of $\Psi_{n-n_2}^*$
where 
$$\Psi_{n}^*\defin \sum_{r=1}^P\sum_{k=0}^{P-1}A_1^rA_{r-k+1}^r\sum_{j=2}^{n}\phi^{-j}u_{jP-r+k}^*$$
and the sequence $\{u_k^*\}$ is defined in part 2(v) of the proof of Lemma~\ref{lemm:Cn}. The sequence  $\{u_k^*\}$ is independent with respect to $X_0$ and $\{u_k\}$, thus the sequence $\{\Psi_n\}$ is also independent with respect to $X_0$ and $\{u_k\}$. Since the sequence $\{u_k^*\}$ is centered periodically distributed with second order moments, the sequence $\{\Psi_{n}^*\}$  converges almost surely and in quadratic mean to some random variable $\zeta^*$. Thanks to the definition~(\ref{Psi}) of $\Psi_{n_2}^n$  and the periodicity of the distribution of the innovation $\{u_k\}$, we deduce the distribution of $\zeta^*$.  Then limit~(\ref{lim:Cn}) is proved.
\hfill{\small $\Box$} 
\end{proof}


%

\begin{thebibliography}{}
{\small

\bibitem{A.G:95}
Adams G.J., Goodwin G.C. (1995)
Parameter estimation for periodic ARMA models.
\emph{J. Time Ser. Anal.} 16(2): 127--147.

\bibitem{AA:12a}
Aknouche A. (2012a)
Implication of instability on econometric and financial time series modeling.
In \emph{Econometrics : New Research} (editors : Mendez S.A., Vega A.M.). Nova Publishers, New York. 

\bibitem{AA:12b}
Aknouche A. (2012b)
Multi-stage weighted least squares estimation of ARCH processes in the stable and unstable cases.
\emph{Statist. Inference Stochast. Process.} 15: 241--256.

\bibitem{AA:13}
Aknouche A. (2013)
Knife edge effect in strong periodic autoregressions.
Preprint.

\bibitem{A.A:12}
Aknouche A, Al-Eid E. (2012)
Asymptotic inference of unstable periodic ARCH processes.
\emph{Statist. Inference Stochast. Process.} 15: 61--79.

\bibitem{A.B:09}
Aknouche A, Bibi A (2009)
Quasi-maximum likelihood estimation of periodic GARCH and periodic ARMA-GARCH processes.
\emph{J. Time Ser. Anal.} 30: 19--46.

\bibitem{ATW:59}
Anderson T.W. (1959) 
On asymptotic distributions of estimates of parameters of stochastic difference equations. 
\emph{Ann. Math. Statist.} 30: 676--687.

\bibitem{JA:09}
Antoni J. (2009)
Cyclostationarity by Examples.
\emph{Mechan. System.  Signal Process.} 23: 987--1036. 

\bibitem{B.L:01}
Basawa, I.V., Lund R. (2001)
Large sample properties of parameter estimates for periodic ARMA models.
\emph{J. Time Ser. Anal.} 22 (6): 651--663.

\bibitem{B.C:09}
Bittanti S., Colaneri P. (2009)
\emph{Periodic Systems: Filtering and Control}.
Springer-Verlag, New York.

\bibitem{B.H.L:94}
Bloomfield P., Hurd H.L., Lund R.B. (1994)
Periodic correlation in Stratospheric ozonz data.
\emph{J. Time Series Anal.} 15(2): 127--150.

\bibitem{B.F:95}
Boswijk H.P., Franses P.H. (1995)
Testing for periodic integration.
\emph{Economics Letters} 48: 241--248.

\bibitem{B.F:96}
Boswijk H.P., Franses P.H. (1996)
Unit roots for periodic integration.
\emph{J. Time Ser. Anal.} 17: 221--245.

\bibitem{RCB:2005}
Bradley R.C. (2005)
Basic properties of strong mixing conditions. A survey and some open questions.
\emph{Probability Surveys} 2 : 107-144.

\bibitem{C.L.N.SR:14}
Chaari F., Le\'skow J., Napolitano A., Sanchez-Ramirez A. (editors) (2014)
\emph{Cyclostationarity : Theory and Methods}. Lecture Notes in Mechanical Engineering. Springer-Verlag, Cham (Switzerland).

\bibitem{C.M:08}
Collet P., Martinez S. (2008)
Asymptotic velocity of one dimensional diffusions with periodic drift.
\emph{J. Math. Biol.} 56: 765--792.

\bibitem{D.D.L.L.L.P:2007}
Dedecker J., Doukhan P., Lang G., Le{\'o}n R. J.R., Louhichi S., Prieur C. (2007)
\emph{Weak Dependence : with Examples ans Applications}. 
Lecture notes in Statistics 190. Springer-Verlag, New York.

\bibitem{PD:94}
Doukhan P. (1994)
\emph{Mixing : Properties and Examples}.
Lecture Notes in Statistics 85. Springer-Verlag, New York. 

\bibitem{D.J:82}
Dragan Ya., Javors\'{}kyj I. (1982)
\emph{Rhythmics of Sea Waving and Underwater Acoustic Signals}.
Naukova dumka, Kiev (Kijev) (in Russian).

\bibitem{F.R.S:11}
Francq C., Roy R., Saidi A. (2011)
Asymptotic properties of weighted least squares estimation in weak PARMA models.
\emph{J. Time Ser. Anal.} 32, 699--723.

\bibitem{F.P:04}
Franses P., Paap R. (2004)
\emph{Periodic Time Series}.
Oxford University Press, Oxford.

\bibitem{G.N.P:06}
Gardner W.A.,  Napolitano A. Paura, L. (2006) 
Cyclostationarity : half a century of research. 
\emph{Signal Processing} 86:  639--697.

\bibitem{G.H.L:96}
Ghysels E., Hall A., Lee H.S. (1996)
On periodic structures and testing for seasonal unit root.
\emph{J. Amer. Statist. Associat.} 91: 1551--1559.

\bibitem{H.M:02}
Hurd H.L., Makagon A., Miamee A.G. (2002)
On AR(1) models with periodic and almost periodic coefficients.
\emph{Stochastic Process. Appl.} 100: 167--185.

\bibitem{M.M:98}
Monsour J.M., Mikulski P.W.  (1998)
On Limiting Distributions in Explosive Autoregressive Processess.
\emph{Statistics \& Probability Letters} 37: 141--147.

\bibitem{OCSB:88}
Osborn D.R, Chui A.P.L., Smith J.P., Birchenhall C.R. (1988)
Seasonality and the order of integration for consumption.
\emph{Oxford Bull. Econom. Statist.} 50: 361--377.
 
\bibitem{MP:78}
Pagano M. (1978)
On periodic and multiple autoregression.
\emph{Ann.  Statist.} 6: 1310--1317.

\bibitem{PCBP:1987}
Phillips P.C.B. (1987)
Time series regression with a unit root.
\emph{Econometrica} 55 (2): 277--301.

\bibitem{S.P.S.G:05}
Serpedin E., Pandura F., Sari I.,  Giannakis G.B. (2005)
Bibliography on Cyclostationarity. 
\emph{Signal Processing} 85:  2233--2303.

\bibitem{MR:1956}
Rosenblatt M. (1956)
A central limit theorem and a strong mixing condition.
\emph{Proc. Nat. Acad. Sci. USA} 27 : 832--837.

\bibitem{BPS:74}
Stigum B.P. (1974)
Asymptotic properties of dynamic stochastic parameter estimates (III).
\emph{J. Multivariate Anal.} 4: 351--381.

\bibitem{T.G:80}
Tiao G.C., Grupe M.R. (1980)
Hidden periodic autoregressive-moving average models in time series data.
\emph{Biometrika} 67: 365--373.

\bibitem{MBT:79}
Troutman B.M. (1979)
Some results in periodic autoregressions.
\emph{Biometrika} 66: 219--228.

\bibitem{AV:85}
Vecchia A. (1985)
Maximum likelihood estimation for periodic autoregressive moving average models. \emph{Technometrics} 27: 375--384.

\bibitem{AWvdV:98}
Van der Vaart A.W. (1998)
\emph{Asymptotic Statistics}
Cambridge University Press, Cambridge.
}
\end{thebibliography}
\end{document}